\documentclass[12pt]{amsart}

\usepackage{aliascnt,amsmath,amssymb,amsthm,amsfonts,enumerate,mathrsfs,latexsym,mathtools,bm,xcolor}
\usepackage{stmaryrd}
\usepackage[abbrev]{amsrefs}
\usepackage{ytableau}

\usepackage[OT2,T1]{fontenc}
\usepackage{scalerel}

\usepackage[marginparwidth=0pt,margin=24truemm]{geometry}

\definecolor{mylinkcolor}{RGB}{16, 156, 81}
\definecolor{mycitecolor}{RGB}{20, 80, 140}

\usepackage{hyperref}
\usepackage[nameinlink]{cleveref}
\hypersetup{
    setpagesize=false,
    hypertexnames=false,
    bookmarksnumbered=true,
    bookmarksopen=true,
    colorlinks=true,
    linkcolor=mylinkcolor,
    citecolor=mycitecolor,
}
\usepackage{autonum}
\numberwithin{equation}{section}

\theoremstyle{plain}

\newtheorem{thm}{Theorem}[section]
\crefname{thm}{Theorem}{Theorems}

\newaliascnt{lem}{thm}
\newtheorem{lem}[lem]{Lemma}
\aliascntresetthe{lem}
\crefname{lem}{Lemma}{Lemmas}

\newaliascnt{prop}{thm}
\newtheorem{prop}[prop]{Proposition}
\aliascntresetthe{prop}
\crefname{prop}{Proposition}{Propositions}

\newaliascnt{cor}{thm}
\newtheorem{cor}[cor]{Corollary}
\aliascntresetthe{cor}
\crefname{cor}{Corollary}{Corollaries}

\newaliascnt{fact}{thm}

\aliascntresetthe{fact}
\crefname{fact}{Fact}{Facts}

\newaliascnt{conj}{thm}
\newtheorem{conj}[conj]{Conjecture}
\aliascntresetthe{conj}
\crefname{conj}{Conjecture}{Conjectures}

\newaliascnt{dfn}{thm}

\aliascntresetthe{dfn}
\crefname{dfn}{Definition}{Definitions}

\newaliascnt{rem}{thm}
\newtheorem{rem}[rem]{Remark}
\aliascntresetthe{rem}
\crefname{rem}{Remark}{Remarks}

\newaliascnt{ex}{thm}
\newtheorem{ex}[ex]{Example}
\aliascntresetthe{ex}
\crefname{ex}{Example}{Examples}

\newtheorem{mainthm}{Main Theorem}

\crefname{mainthm}{Main Theorem}{Main Theorems}
\newaliascnt{mainconj}{mainthm}
\newtheorem{mainconj}[mainconj]{Main Conjecture}
\aliascntresetthe{mainconj}
\crefname{mainconj}{Main Conjecture}{Main Conjectures}

\crefname{section}{Section}{Sections}

\allowdisplaybreaks[2]

\everymath{\displaystyle}

\providecommand{\Z}{}\renewcommand{\Z}{\mathbb{Z}}
\providecommand{\Q}{}\renewcommand{\Q}{\mathbb{Q}}

\providecommand{\HH}{}\renewcommand{\HH}{\mathbb{H}}
\newcommand{\M}{\mathcal{M}}
\newcommand{\Mz}{\M_{\Z}}
\newcommand{\Rg}{\mathcal{G}_{\Z}}
\newcommand{\mes}{\mathcal{E}}
\newcommand{\PhiG}{\Phi_{G}}
\newcommand{\Phig}{\Phi_{g}}
\newcommand{\Phib}{\Phi_{\beta}}
\newcommand{\Filw}{\operatorname{Fil}^{W}}
\newcommand{\Lam}{\Lambda}

\newcommand{\GD}{\mathrm{D}}
\newcommand{\GW}{\mathrm{W}}
\newcommand{\Gdelta}{\delta}
\newcommand{\sltwo}{\mathfrak{sl}_2}
\DeclareMathOperator{\SSYT}{SSYT}

\ytableausetup{boxsize=0.7em,aligntableaux=center}
\newcommand{\pes}[1]{G\!\left(\ydiagram{#1}\right)}
\newcommand{\pgs}[1]{g\!\left(\ydiagram{#1}\right)}
\newcommand{\pbs}[1]{\beta\!\left(\ydiagram{#1}\right)}
\newcommand{\vsmall}{\rotatebox[origin=c]{-90}{$<$}}

\title{Schur Eisenstein series and Schur MacMahon series}

\author{Henrik Bachmann}
\address{Graduate School of Mathematics, Nagoya University, Nagoya, Japan.}
\email{henrik.bachmann@math.nagoya-u.ac.jp}

\author{Jinbo Yu}
\address{Graduate School of Mathematics, Nagoya University, Nagoya, Japan.}
\email{jinbo.yu.e6@math.nagoya-u.ac.jp}

\date{\today}


\begin{document}

\begin{abstract}
We introduce and study two partition-indexed families of quasimodular forms obtained from Schur functions: Schur Eisenstein series and Schur MacMahon series. An explicit transition between them can be interpreted as a convolution in a Fa\`a di Bruno Hopf algebra of symmetric functions. We discuss the classical $\sltwo$-action and prove that Schur Eisenstein series for partitions with parts of size at most $3$ give a basis for quasimodular forms. Further, we conjecture that the Schur MacMahon series span all quasimodular forms with integral coefficients.
\end{abstract}

\maketitle

\section{Introduction}\label{sec:intro}

The purpose of this paper is to introduce two families of quasimodular forms indexed by
partitions, which we call Schur MacMahon series and Schur Eisenstein series,
and to explain their relation through symmetric functions. We give an explicit
transition between the two families, describe their differential and
$\sltwo$-structures, construct natural bases of quasimodular forms, and investigate
their integral span.

Quasimodular forms were introduced by Kaneko and Zagier \cite{KZ}. For $\tau$ in the
upper half-plane $\HH$ write $q=e^{2\pi i\tau}$. For even $k\geq2$, the Eisenstein series of weight $k$ is defined by
\begin{equation}\label{eq:Gk}
  G_k(\tau):=-\frac{B_k}{2\,k!}+\frac{1}{(k-1)!}\sum_{n\geq1}\sigma_{k-1}(n)q^n.
\end{equation}
The ring of quasimodular forms for $\operatorname{SL}_2(\Z)$ with rational Fourier
coefficients is
\begin{equation}\label{eq:Mdef}
\M:=\Q[G_2,G_4,G_6]=\bigoplus_{k\geq0}\M_k,
\end{equation}
where $\M_k$ denotes the space of weight $k$ quasimodular forms. 

Another family of quasimodular forms is given by MacMahon's sums-of-divisors
\cite{McM}
\[
  A_r(q):=\sum_{0<m_1<\cdots<m_r}
  \prod_{j=1}^r\frac{q^{m_j}}{(1-q^{m_j})^2}
  \qquad(r\geq1).
\]
For example, $A_1=G_2+1/24$. Andrews and Rose proved that all $A_r$ are
quasimodular forms of mixed weight \cite{AR}. For example, we have
\[
  A_2=\frac12(G_2^2-G_4)+\frac18G_2+\frac3{640} \in \M_4 + \M_2 + \M_0. 
\]

Multiple Eisenstein series were introduced by Gangl--Kaneko--Zagier
\cite{GKZ}. They are higher-depth analogues of
classical Eisenstein series whose constant terms are multiple zeta values. In \cite{B2},
the first author showed that the all-$2$ series provide the natural homogeneous
corrections of MacMahon's mixed-weight series. More precisely, in our normalization,
\[
  G_{\{2\}^r}\in \M_{2r}
\]
is the weight-$2r$ part of $A_r$, and the two families are related by an explicit
triangular change of variables. For example,
\[
  G_{2,2}=A_2-\frac18A_1+\frac1{1920}
  =\frac12(G_2^2-G_4) \in  \M_{4}.
\]

In this paper, we generalize these objects and their relationships from columns to arbitrary Young diagrams.
The index $\{2\}^r=(2,\dots,2)$ corresponds to a column of height $r$. For example,
for the shape $\lambda=\ydiagram{2,1}$ the corresponding Schur MacMahon series is
\[
\begin{aligned}
  g\!\left(\ydiagram{2,1}\right)
  &=
  \sum_{
    \arraycolsep=1.4pt\def\arraystretch{0.8}
    {\footnotesize\begin{array}{ccc}
      & m_2 & \leq m_3\\
      & \vsmall &\\
      & m_1 &
    \end{array}}
  }
  \frac{q^{m_1+m_2+m_3}}
  {(1-q^{m_1})^2(1-q^{m_2})^2(1-q^{m_3})^2}.
\end{aligned}
\]
Here $m_1,m_2,m_3$ are positive integers. The horizontal inequality is weak and the
vertical inequality is strict.

This generalization is inspired by the Schur multiple
zeta values of Nakasuji--Phuksuwan--Yamasaki \cite{NPY}, which attach zeta values to
Young tableaux,
and by the second author's recent combination of Schur multiple zeta values and multiple Eisenstein
series into Schur multiple Eisenstein series \cite{Yu}. We concentrate on the simplest
uniform case, in which every entry of the tableau is set equal to $2$. The construction
of the second author then gives one series for each partition, while evaluating the same
Schur shape in MacMahon's variables gives another. These are the two new families introduced here.
The guiding idea is the all-shape analogue of the result above: the Schur Eisenstein
series $G(\lambda)$ is the weight-homogeneous correction of the Schur MacMahon series
$g(\lambda)$. We also mention that there are other natural ways of assigning
quasimodular forms to partitions, for example the partition Eisenstein series of
Amdeberhan--Griffin--Ono--Singh \cite{AGOS}. In \cref{rem:AGOS} we compare these
with our objects and explain how the Schur Eisenstein series can be viewed as
distinguished character traces of partition Eisenstein series.

These objects fit naturally into the ring $\Lam$ of symmetric functions over $\Q$ in
variables $x_1,x_2,\dots$. For $j\geq1$, let
\[
  p_j=\sum_{i\geq1}x_i^j
\]
be the $j$th power-sum symmetric function, so that
$\Lam=\Q[p_1,p_2,\dots]$. We write $e_r$ for the elementary symmetric functions and
$s_\lambda$ for the Schur functions, and use the Hall inner product for which the
Schur functions are orthonormal. The three specializations central to this work are
the algebra homomorphisms
  \[
  \Phib:\Lam\longrightarrow\Q,\qquad
  \Phig:\Lam\longrightarrow\M,\qquad
  \PhiG:\Lam\longrightarrow \M,
\]
specified for $r\geq1$ by
\[
  \Phib(p_r)=\beta_{2r}:=\frac{\zeta(2r)}{(2\pi i)^{2r}}
      =-\frac{B_{2r}}{2(2r)!},\qquad
  \Phig(p_r)=\sum_{m\geq1}\left(\frac{q^m}{(1-q^m)^2}\right)^r,
  \qquad
  \PhiG(p_r)=G_{2r},
\]
where $G_{2r}$ is the renormalized Eisenstein series \eqref{eq:Gk}. We introduce
the three partition-indexed families in parallel by
\begin{equation}\label{eq:introthreefamilies}
  \beta(\lambda):=\Phib(s_\lambda),\qquad
  g(\lambda):=\Phig(s_\lambda),\qquad
  G(\lambda):=\PhiG(s_\lambda).
\end{equation}
Thus
\[
  \beta(\lambda)
  =\frac{\zeta\bigl(\{2\}^{\lambda}\bigr)}{(2\pi i)^{2|\lambda|}}\in\Q
\]
is the renormalized Schur multiple zeta value of the all-$2$ tableau, while
\[
  g(\lambda)
  =s_\lambda\left(\frac{q}{(1-q)^2},\frac{q^2}{(1-q^2)^2},\dots\right)
  \in\M\cap\Z[[q]]
\]
is what we call the \emph{Schur MacMahon series} of shape $\lambda$. Single columns
give the MacMahon series $g((1^r))=A_r$.

Finally, $G(\lambda)\in\Q[[q]]$ is the renormalized all-$2$
specialization of the Schur multiple Eisenstein series of the second author \cite{Yu}.
We call it the \emph{Schur Eisenstein series} of shape $\lambda$. Thus $g(\lambda)$ is
quasimodular of mixed weight at most $2|\lambda|$, whereas $G(\lambda)$ is homogeneous
of weight $2|\lambda|$ and has constant term $\beta(\lambda)$.

The three families therefore encode the constant, filtered and homogeneous aspects
of the same Schur-function construction. Both $g(\lambda)$ and $G(\lambda)$ span
$\M$ over $\Q$, but in complementary ways: the former are integral and filtered,
whereas the latter are rational and graded. Our first main result makes the passage
between them completely explicit. Here $T(2m,2j)\in\Z$ denote the central factorial
numbers of the second kind (\cref{subsec:transition}).

\begin{mainthm}\label{mainthm:transition}
\begin{enumerate}[(i)]
    \item The algebra automorphism $\theta$ of $\Lam$ defined by
\[
  \theta(p_m)=\beta_{2m}+\sum_{j=1}^{m}\frac{m}{j}\,\frac{(2j)!}{(2m)!}\,T(2m,2j)\,p_j
  \qquad(m\geq1)
\]
satisfies $\PhiG=\Phig\circ\theta$. 
\item Every Schur Eisenstein series can be written as
\[
  G(\lambda)
  = \beta(\lambda)+\sum_{0<|\mu|<|\lambda|}d_\lambda(\mu)\,g(\mu)+g(\lambda),
\]
where each coefficient $d_\lambda(\mu)$ is an integer linear combination of the
renormalized Schur multiple zeta values $\beta(\nu)$ with $\nu\vdash|\lambda|-|\mu|$.
\end{enumerate}

\end{mainthm}

For example, for the shape $\lambda=(2,1)$ we have
\begin{equation}\label{eq:introG21}
\begin{aligned}
  \pes{2,1}
  &=\pbs{2,1}-2\,\pbs{1,1}\,\pgs{1}
    +3\,\pbs{1}\,\pgs{2}-\pbs{1}\,\pgs{1,1}+\pgs{2,1}.
\end{aligned}
\end{equation}

The full statement is \cref{thm:transition}, which also inverts $\theta$ in terms of
central factorial numbers of the first kind, the integrality of the coefficients being
\cref{prop:zbeta}. The single-column case of the transition is the result of
\cite{B2} recalled above, and the single-row case is essentially a result of
Amdeberhan--Ono--Singh \cite[Thm.~1.4]{AOS}.
The expansion is supported on the subdiagrams
$\mu\subseteq\lambda$ (\cref{thm:subdiagram-support}). We conjecture that every
such subdiagram actually occurs for a straight shape (\cref{conj:subdiagram}). The
kernels of both specializations are computed in \cref{thm:kernels}. The transition
also ties the three specializations together
Hopf-algebraically: for an explicit connected graded deformed Fa\`a di Bruno coproduct
$\Delta_{\!F}$ on $\Lam$ one has, with $m$ denoting multiplication in $\M$,
\[
  \PhiG=m\circ(\Phib\otimes\Phig)\circ\Delta_{\!F},
\]
so the Schur Eisenstein series are the convolution of the Schur multiple zeta
values with the Schur MacMahon series and $\theta$ is the associated winding
automorphism (\cref{thm:fdb}). The expansion \eqref{eq:introG21} is obtained
directly from this convolution in \cref{ex:fdb21}. This is the analogue, for
the two families, of the Goncharov-coproduct description of the Fourier expansion
of multiple Eisenstein series by the first author and Tasaka \cite{BT}.

The second main result concerns the $\sltwo$-structure of $\M$. The ring $\M$ carries
the classical $\sltwo$-triple $(\GD,\GW,\Gdelta)$, where $\GD=q\frac{d}{dq}$, $\GW$ is
the weight grading and $\Gdelta=-\frac12\frac{\partial}{\partial G_2}$. On $\Lam$ consider
the derivations
\[
  \delta_\Lam=-\frac12\frac{\partial}{\partial p_1},\qquad
  W_\Lam=2\sum_{n\geq1}np_n\frac{\partial}{\partial p_n},\qquad
  D_\Lam(e_r)=3 e_1 e_r-(r+1)(2r+3)e_{r+1}.
\]
The last identity uniquely determines $D_\Lam$ because $\Lam=\Q[e_1,e_2,\dots]$.
Recall that a triple $(D,W,\delta)$ is an $\sltwo$-triple if
\[
  [W,D]=2D,\qquad [W,\delta]=-2\delta,\qquad [\delta,D]=W,
\]
where $[X,Y]=X\circ Y-Y\circ X$.

\begin{mainthm}\label{mainthm:sl2}
The triple $(D_\Lam,W_\Lam,\delta_\Lam)$ is an $\sltwo$-triple and we have:
\[
  \PhiG\circ D_\Lam=\GD\circ\PhiG,\qquad
  \PhiG\circ W_\Lam=\GW\circ\PhiG,\qquad
  \PhiG\circ\delta_\Lam=\Gdelta\circ\PhiG.
\]
Consequently, for every partition $\lambda$,
\[
  \Gdelta\, G(\lambda)=-\frac12\sum_{\mu=\lambda-\square}G(\mu),
\]
the sum running over all partitions obtained from $\lambda$ by removing one box, and
\[
  \GD\,G(\lambda)=\PhiG\bigl(D_\Lam s_\lambda\bigr).
\]
In particular, if $\lambda$ has $s$ columns, then $D_\Lam(s_\lambda)$ is an explicit
linear combination of Schur functions with at most $s+1$ columns.
\end{mainthm}

For a single column $\lambda=(1^r)$ the derivative formula collapses to 
\[
  \GD\, G_{\{2\}^r}=3\,G_2\,G_{\{2\}^r}-(r+1)(2r+3)\,G_{\{2\}^{r+1}}
  \qquad(r\geq0),
\]
so that on columns both operators act by removing and adding a single box, e.g.
\[
  \Gdelta\,\pes{1,1,1} = -\frac12\,\pes{1,1},
  \qquad\quad
  \GD\,\pes{1,1} = 3\,\pes{1}\,\pes{1,1}-21\,\pes{1,1,1}.
\]
The general straight-shape formula is given explicitly in \cref{cor:Dexplicit}, while
\cref{thm:sl2-general} gives its skew-shape determinantal version. Conjugating the symmetric-function triple above by the winding
automorphism $\theta$ gives the corresponding triple adapted to $g$ and lifts the
Andrews--Rose derivative (\cref{prop:widthD,prop:macmahontriple}).

The third main result gives a basis of each weight space of $\M$. The partitions of
$n$ with parts of size at most $3$ are exactly as numerous as the monomials
$G_2^aG_4^bG_6^c$ of weight $2n$.

\begin{mainthm}\label{mainthm:basis}
For every $n\geq 0$ the Schur Eisenstein series
\[
  \bigl\{\,G(\lambda)\ \big|\ \lambda\vdash n,\ \lambda_1\leq3\,\bigr\}
\]
form a basis of the space $\M_{2n}$ of weight $2n$ quasimodular forms.
\end{mainthm}

\begin{samepage}
For example, Ramanujan's Delta function
\[
  \Delta=q\prod_{m\geq1}(1-q^m)^{24}\in\M_{12}
\]
can be written as
\[
\begin{aligned}
  \frac{1}{9!}\Delta
  &=5\,\pes{3,3}-5\,\pes{3,2,1}+5\,\pes{3,1,1,1}-59\,\pes{2,2,2}\\
  &\quad+64\,\pes{2,2,1,1}-69\,\pes{2,1,1,1,1}
  +69\,\pes{1,1,1,1,1,1}.
\end{aligned}
\]
\end{samepage}

The proof (\cref{thm:equiv}) is $2$-adic: the Weierstrass differential equation shows
that every Eisenstein polynomial $R_m(G_4,G_6)$, $m\geq4$, is divisible by $2$ over
$\Z_{(2)}$, and M\"obius inversion on set partitions then makes the transition matrix
to the monomial basis invertible modulo $2$.

We also consider the integral lattice
\[
  \Mz:=\M\cap\Z[[q]].
\]
For this question we use the usual normalization
\begin{equation}\label{eq:Ek}
  E_k:=-\frac{2\,k!}{B_k}G_k
  =1-\frac{2k}{B_k}\sum_{n\geq1}\sigma_{k-1}(n)q^n
  \qquad(k\geq2\ \text{even}).
\end{equation}
Thus $E_2=-24G_2$, $E_4=1440G_4$, $E_6=-60480G_6$, and
$\M=\Q[E_2,E_4,E_6]$. The Schur MacMahon series generate the subring
\[
  \Rg:=\sum_\lambda\Z\,g(\lambda)=\Z[A_1,A_2,A_3,\dots]\subseteq\Mz.
\]
The inclusion $\Z[E_2,E_4,E_6]\subseteq\Mz$ is strict, and $\Mz$ is not graded by
weight. For example,
\[
  \frac{1-E_2}{24}=A_1=\sum_{n\geq1}\sigma_1(n)q^n
\]
is integral, although neither of its two homogeneous components is integral. We
conjecture that the inclusion $\Rg\subseteq\Mz$ is an equality.

\begin{mainconj}\label{mainconj:integralspan}
The Schur MacMahon series span the integral lattice over $\Z$, i.e.
\[
  \Rg =\Mz .
\]
\end{mainconj}

The paper is organized as follows. \cref{sec:schurmacmahon} defines the three
specializations $\beta$, $g$ and $G$,
proves the transition theorem (\cref{mainthm:transition}) and describes the kernels.
\cref{sec:sl2} proves the $\sltwo$-theorems (\cref{mainthm:sl2}). \cref{sec:bases}
proves the basis theorem (\cref{mainthm:basis}), formulates the integral spanning
conjecture and records computational evidence. All numerical claims have been checked
with SageMath.

\section*{Acknowledgments}
We would like to thank William Craig for giving helpful comments on the integrality
of quasimodular forms.
The second author was financially supported by JST SPRING, Grant Number JPMJSP2125.
This project was partially supported by JSPS KAKENHI Grant 26K22254.

\section{The two families and the transition theorem}\label{sec:schurmacmahon}

\subsection{Symmetric functions}\label{subsec:symfct}
In this subsection we recall some basic facts on symmetric functions. All of them,
together with proofs and further details, can be found in the standard references
\cite{Mac} and \cite[Ch.~7]{St}.
Let $\Lam$ denote the ring of symmetric functions over $\Q$ in variables
$x_1,x_2,\dots$, with elementary symmetric functions $e_l$, complete homogeneous
symmetric functions $h_l$, power sums $p_m=\sum_{i\geq1}x_i^m$ and Schur functions
$s_\lambda$. Recall that
$$\Lam=\Q[e_1,e_2,\dots]=\Q[p_1,p_2,\dots],$$ that the generating series
\begin{equation}\label{eq:newton}
  E(t)=\sum_{l\geq0}e_l\,t^l=\exp\Bigl(\sum_{j\geq1}\frac{(-1)^{j-1}}{j}\,p_j\,t^j\Bigr),
  \qquad
  H(t)=\sum_{l\geq0}h_l\,t^l=\frac{1}{E(-t)},
\end{equation}
encode the Newton relations, and that the Jacobi--Trudi formula expresses skew
Schur functions as determinants in the $e_l$,
\begin{equation}\label{eq:jt}
  s_{\lambda/\mu}=\det\bigl(e_{\lambda'_i-\mu'_j-i+j}\bigr)_{1\leq i,j\leq s},
\end{equation}
where $\lambda',\mu'$ denote the conjugate partitions, $s=\lambda_1$, and $e_0=1$,
$e_l=0$ for $l<0$. Power sums expand into hooks,
\begin{equation}\label{eq:hooks}
  p_n=\sum_{\substack{a+b=n\\ a\geq1,\,b\geq0}}(-1)^{b}\,s_{(a,\{1\}^b)}.
\end{equation}
Finally, the partial derivative $\partial/\partial p_1$ (taken in the presentation
$\Lam=\Q[p_1,p_2,\dots]$) is a derivation of $\Lam$ and acts on Schur functions by
removing a box,
\begin{equation}\label{eq:boxremoval}
  \frac{\partial}{\partial p_1}\,s_\lambda=\sum_{\mu=\lambda-\square}s_\mu,
\end{equation}
where the sum runs over all partitions $\mu$ obtained from $\lambda$ by removing one
box. (In the language of the Hall inner product this operator is the skewing operator
$p_1^\perp$.)

We will measure shapes by their number of columns. For $n\geq0$ let
\begin{equation}\label{eq:widthfil}
  \Filw_n\Lam:=\bigl\langle s_\lambda\ \big|\ \lambda_1\leq n\bigr\rangle_\Q
\end{equation}
be the span of the Schur functions of \emph{width} at most $n$. By the 
Jacobi--Trudi formula \eqref{eq:jt}, $\Filw_n\Lam$ is also the span of the products
$e_{m_1}\cdots e_{m_c}$ of at most $n$ elementary symmetric functions, since
conversely every such product is a sum of Schur functions whose shapes are unions of
$c\leq n$ vertical strips. In particular
\[
  \Filw_0\Lam=\Q\subset\Filw_1\Lam\subset\Filw_2\Lam\subset\cdots,
  \qquad
  \Filw_m\Lam\cdot\Filw_n\Lam\subseteq\Filw_{m+n}\Lam,
\]
so the width defines an algebra filtration on $\Lam$, with $\Filw_1\Lam$ spanned by
the constants and the single columns $e_l$.

\subsection{Schur MacMahon series}\label{subsec:defg}
For $m\geq 1$ set
\[
  x_m := \frac{q^m}{(1-q^m)^2} \in \Z[[q]].
\]
For a partition $\lambda=(\lambda_1\geq\dots\geq\lambda_k>0)$ with Young diagram
$$D(\lambda)=\{(i,j)\mid 1\leq i\leq k,\ 1\leq j\leq\lambda_i\}$$ define the
\emph{Schur MacMahon series}
\begin{equation}\label{eq:glambda}
  g(\lambda)
  := \sum_{M\in\SSYT(\lambda)}\ \prod_{(i,j)\in D(\lambda)}\frac{q^{m_{ij}}}{\bigl(1-q^{m_{ij}}\bigr)^2},
\end{equation}
where the sum is over all semistandard Young tableaux $M=(m_{ij})$ of shape $\lambda$ with
entries in $\Z_{>0}$ (weakly increasing along rows, strictly increasing down columns).
Equivalently $g(\lambda)$ is the Schur-function specialization
\[
  g(\lambda) = s_\lambda(x_1,x_2,x_3,\dots),\qquad g(\varnothing) = 1.
\]
For a single column $\lambda=(1^r)$ this recovers the classical MacMahon $q$-series
\begin{equation}\label{eq:Ar}
  g((1^r)) = e_r(x_1,x_2,\dots)
  = \sum_{m_1>\dots>m_r>0}\frac{q^{m_1+\dots+m_r}}{(1-q^{m_1})^2\cdots(1-q^{m_r})^2}
  =: A_r.
\end{equation}
In particular, $A_0=1$. The specialization, whose image lies in $\M$ by \cref{prop:RgInMz},
\[
  \Phig:\Lam \longrightarrow \M,\qquad f\longmapsto f(x_1,x_2,x_3,\dots),
\]
is well defined $q$-adically, restricts to $\Lam_\Z\to\Z[[q]]$ on the symmetric
functions with integer coefficients, and satisfies $\Phig(s_\lambda)=g(\lambda)$ and
$\Phig(e_r)=A_r$. Since the Schur functions form a $\Z$-basis of $\Lam_\Z$, the image of
$\Lam_\Z$ is
\[
  \Rg := \Phig(\Lam_\Z) = \sum_{\lambda}\Z\,g(\lambda) = \Z[A_1,A_2,A_3,\dots],
\]
the second equality using $\Lam_\Z=\Z[e_1,e_2,\dots]$ and \eqref{eq:Ar}.

\begin{prop}\label{prop:RgInMz}
One has
\[
  \Z[E_2,E_4,E_6] \,\, \subseteq \,\, \Rg \,\, \subseteq \,\, \Mz,
\]
and hence $\Rg\otimes_\Z\Q = \M$.
\end{prop}

\begin{proof}
Let $p_r$ be the $r$-th power sum and put
\begin{equation}\label{eq:Pr}
  P_r:=\Phig(p_r)=\sum_{m\geq 1}x_m^r=\sum_{m\geq 1}\frac{q^{rm}}{(1-q^m)^{2r}}
  =\sum_{n\geq 1}\Bigl(\sum_{d\mid n}\binom{d+r-1}{2r-1}\Bigr)q^n.
\end{equation}
Since
\begin{equation}\label{eq:binomcentral}
  \binom{d+r-1}{2r-1}=\frac{d(d^2-1)(d^2-2^2)\cdots\bigl(d^2-(r-1)^2\bigr)}{(2r-1)!}
\end{equation}
is an odd polynomial in $d$ of degree $2r-1$, the coefficient $\sum_{d\mid n}\binom{d+r-1}{2r-1}$
is a $\Q$-linear combination of the divisor sums $\sigma_1(n),\sigma_3(n),\dots,\sigma_{2r-1}(n)$. 
Since
\begin{equation}
    \sum_{n\geq1}\sigma_{2j-1}(n)q^n=(2j-1)!(G_{2j}(\tau)-\beta_{2j}),\quad (j\geq1)
\end{equation}
we have $P_r\in \Q[G_2,G_4,G_6,\dots]= \M$ for all $r\geq 1$. As every Schur function is a
$\Q$-polynomial in the power sums, $g(\lambda)\in \M$, and the tableau definition
\eqref{eq:glambda} shows directly that $g(\lambda)\in\Z[[q]]$. Hence
$g(\lambda)\in\Mz$ and $\Rg\subseteq\Mz$.

The first three power sums are
\[
  P_1=\frac{1-E_2}{24},\quad
  P_2=\frac{E_4+10E_2-11}{1440},\quad
  P_3=\frac{191-168E_2-21E_4-2E_6}{120960},
\]
which invert to
\[
  E_2=1-24P_1,\quad
  E_4=1+240P_1+1440P_2,\quad
  E_6=1-504P_1-15120P_2-60480P_3.
\]
Since $P_1,P_2,P_3\in \Rg$, we get $E_2,E_4,E_6\in \Rg$, i.e.\ $\Z[E_2,E_4,E_6]\subseteq \Rg$.
With $\M=\Q[E_2,E_4,E_6]$ this gives $\Rg\otimes_\Z\Q=\M$.
\end{proof}

\subsection{Multiple Eisenstein series and Schur Eisenstein series}\label{subsec:pes}
For even $k\geq 2$ the Eisenstein series of weight $k$ is
\begin{equation}\label{eq:defgk}
  \mathbb{G}_k(\tau)
  = \frac12\sum_{\substack{m,n\in\Z\\(m,n)\neq(0,0)}}\frac{1}{(m\tau+n)^k}
  = \zeta(k) + \frac{(-2\pi i)^k}{(k-1)!}\sum_{n\geq 1}\sigma_{k-1}(n)q^n,
\end{equation}
using Eisenstein summation when $k=2$. Recall the multiple zeta values
\[
  \zeta(k_1,\dots,k_r)=\sum_{0<m_1<\dots<m_r}\frac{1}{m_1^{k_1}\cdots m_r^{k_r}}
  \qquad(k_1,\dots,k_{r-1}\geq1,k_r\geq 2),
\]
and the multiple Eisenstein series of Gangl--Kaneko--Zagier \cite{GKZ}, defined for
$k_1,\dots,k_r\geq 2$ by\footnote{Eisenstein summation is used if $k_r=2$. The order is
reversed compared with \cite{GKZ}.}
\[
  \mathbb{G}_{k_1,\dots,k_r}(\tau)
  = \sum_{\substack{0\prec\lambda_1\prec\dots\prec\lambda_r\\ \lambda_i\in\Z\tau+\Z}}
    \frac{1}{\lambda_1^{k_1}\cdots\lambda_r^{k_r}},
\]
where $\prec$ is the lexicographic order on $\Z\tau+\Z$ given by
$m_1\tau+n_1\prec m_2\tau+n_2 \Leftrightarrow m_1<m_2$, or $m_1=m_2$ and $n_1<n_2$.
By \cite{GKZ} ($r=2$) and \cite{B} ($r\geq2$) these admit a Fourier expansion
\[
  \mathbb{G}_{k_1,\dots,k_r}(\tau) = \zeta(k_1,\dots,k_r) + \sum_{n\geq 1}a_n^{k_1,\dots,k_r}q^n,
  \qquad a_n^{k_1,\dots,k_r}\in\mathcal Z[\pi i],
\]
where $\mathcal Z$ denotes the $\Q$-algebra of multiple zeta values. We write
\[
  \mes = \bigl\langle \mathbb{G}_{k_1,\dots,k_r}\mid r\geq 0,\ k_1,\dots,k_r\geq 2\bigr\rangle_\Q,
  \qquad \mathbb{G}_\varnothing = 1,
\]
for the $\Q$-vector space they span. In contrast with the case $r=1$, $k_1$ even, one
expects that in general
$(-2\pi i)^{-(k_1+\dots+k_r)}\,\mathbb{G}_{k_1,\dots,k_r}$
does not lie in $\Q[[q]]$.

The second author generalized these series from tuples to tableaux in \cite{Yu}. For a skew shape
$\lambda/\mu$ and a filling $\mathbf k$ of $D(\lambda/\mu)$ with integers $\geq2$, the
\emph{Schur multiple Eisenstein series} is
\[
  \mathbb{G}(\mathbf k)
  = \lim_{M\to\infty}\lim_{N\to\infty}
    \sum_{(\lambda_{ij})\in\SSYT(\lambda/\mu,\,X^{\tau,>0}_{M,N})}\
    \prod_{(i,j)\in D(\lambda/\mu)}\frac{1}{\lambda_{ij}^{k_{ij}}},
\]
where $X^{\tau,>0}_{M,N}$ is the finite set of lattice points
$m\tau+n\succ0$ with $|m|\leq M$, $|n|\leq N$, ordered by $\prec$, and semistandard
means weakly increasing along rows and strictly increasing down columns with respect
to $\prec$. Single columns recover the $\mathbb{G}_{k_1,\dots,k_r}$ above. We refer to
the work of the second author \cite{Yu} for the convergence statement and basic
properties. For even $k\geq2$ let
$G_k=(2\pi i)^{-k}\mathbb{G}_k$ as in
\eqref{eq:Gk}, set $G_{\{2\}^l}=(2\pi i)^{-2l}\mathbb{G}_{\{2\}^l}$ for the columns, and
define the \emph{Schur Eisenstein series} of a (possibly skew) shape $\lambda/\mu$ as
\[
  G(\lambda/\mu) = (2\pi i)^{-2|\lambda/\mu|}\,\mathbb{G}\bigl(\{2\}^{\lambda/\mu}\bigr),
\]
where $\{2\}^{\lambda/\mu}$ denotes the tableau of shape $\lambda/\mu$ filled entirely
with $2$s. Thus $G\bigl((1^l)\bigr)=G_{\{2\}^l}$, and the columns of the two families
are the MacMahon series \eqref{eq:Ar} on the $g$-side and the all-$2$ multiple
Eisenstein series on the $G$-side.

\begin{prop}[Andrews--Rose \cite{AR}]\label{prop:alderiv}
Each $A_r$ is a quasimodular form of mixed weight, $A_r\in\M$, and for $r\geq 0$,
\[
  q\frac{d}{dq}A_r(q) = 3A_1(q)A_r(q)-(2r+3)(r+1)A_{r+1}(q)+\frac{(r+1)r}{2}A_r(q).
\]
\end{prop}

We use the following structural results of the second author \cite{Yu}. The first is
the Jacobi--Trudi formula.

\begin{thm}\label{thm:jt}
For every skew shape $\lambda/\mu$ with $s=\lambda_1$ columns,
\[
  G(\lambda/\mu)=\det\!\bigl[G_{\{2\}^{\lambda_i'-\mu_j'-i+j}}\bigr]_{1\leq i,j\leq s},
\]
with the conventions $G_{\{2\}^0}=1$ and $G_{\{2\}^l}=0$ for $l<0$. In particular,
$G(\lambda/\mu)\in \M_{2|\lambda/\mu|}$, i.e.\ each Schur Eisenstein series is a
homogeneous quasimodular form of weight $2|\lambda/\mu|$.
\end{thm}

Here the homogeneity follows from the determinant together with the fact that each
column series $G_{\{2\}^l}$ is homogeneous of weight $2l$, which is a consequence of the
exponential identity below. Comparing \eqref{eq:jt} with \cref{thm:jt} shows that
the Schur Eisenstein series are also a specialization of symmetric functions: define the
$\Q$-algebra homomorphism
\[
  \PhiG:\Lam\longrightarrow \M,\qquad \PhiG(e_l):=G_{\{2\}^l}\quad(l\geq1).
\]
Then \cref{thm:jt} says precisely that
\[
  \PhiG(s_{\lambda/\mu}) = G(\lambda/\mu)
\]
for all skew shapes. On power sums $\PhiG$ takes the following simple form.

\begin{lem}\label{lem:psipn}
For all $n\geq1$ one has $\PhiG(p_n)=G_{2n}$.
\end{lem}

\begin{proof}
By a result of the second author \cite[Cor.~4.11]{Yu} (see also
\cite[Eq.~(3.1)]{B2}, both going back to the quasi-shuffle exponential of
Hoffman--Ihara \cite{HI}),
\[
  1+\sum_{l\geq 1}\mathbb{G}_{\{2\}^l}\,T^{l}
  = \exp\Bigl(\sum_{j\geq 1}\frac{(-1)^{j-1}}{j}\,\mathbb{G}_{2j}\,T^{j}\Bigr).
\]
After renormalizing ($T\rightsquigarrow(2\pi i)^{-2}T$) this is the identity
$\sum_l\PhiG(e_l)T^l=\exp\bigl(\sum_j\tfrac{(-1)^{j-1}}{j}\PhiG(p_j)T^j\bigr)$ with
$\PhiG(p_j)=G_{2j}$, which by \eqref{eq:newton} characterizes the images of the power
sums under the algebra homomorphism $\PhiG$.
\end{proof}

Applying $\PhiG$ to the hook expansion \eqref{eq:hooks} immediately gives a proof of the
following expansion of the Eisenstein series in Schur Eisenstein series, which is the
image of the hook formula of the second author \cite[Thm.~2.18]{Yu}.

\begin{prop}\label{prop:G2n}
For $n\geq 1$,
\[
  G_{2n}=\sum_{\substack{a+b=n\\a\geq 1,\,b\geq 0}}(-1)^b\,G\bigl((a,\{1\}^b)\bigr).
\]
\end{prop}

For example,
\[
  G_2=\pes{1},\qquad G_4=\pes{2}-\pes{1,1},\qquad G_6=\pes{3}-\pes{2,1}+\pes{1,1,1}.
\]
Since $\M=\Q[G_2,G_4,G_6]$, this shows in particular that every quasimodular form is a
$\Q$-linear combination of Schur Eisenstein series. Skew shapes reduce to non-skew
shapes by the Littlewood--Richardson rule, applied through $\PhiG$. This also follows
from a result of the second author \cite[Cor.~3.12]{Yu},
\begin{equation}\label{eq:LR}
  G(\lambda/\mu)=\sum_{\nu}c^{\lambda}_{\mu\nu}\,G(\nu),
\end{equation}
with Littlewood--Richardson coefficients $c^\lambda_{\mu\nu}\in\Z_{\geq0}$.

\begin{rem}\label{rem:AGOS}
A complementary way of assigning quasimodular forms to partitions was introduced
by Amdeberhan--Griffin--Ono--Singh \cite{AGOS}. For
$\rho=(1^{m_1}2^{m_2}\cdots)\vdash n$, they define
\[
  \widetilde{G}_\rho:=\prod_{j\geq1}\widetilde{G}_{2j}^{\,m_j},
  \qquad
  \widetilde{G}_{2j}:=-\frac{B_{2j}}{2j}+2\sum_{\ell\geq1}\sigma_{2j-1}(\ell)q^\ell
  =2\,(2j-1)!\,G_{2j},
\]
and study the weighted traces
$\operatorname{Tr}_n(\phi):=\sum_{\rho\vdash n}\phi(\rho)\,\widetilde{G}_\rho$.
By the Frobenius character formula and \cref{lem:psipn}, for $\lambda\vdash n$,
\[
  G(\lambda)=\operatorname{Tr}_n(\phi_\lambda),
  \qquad
  \phi_\lambda(\rho)
  =\frac{\chi^\lambda(\rho)}{z_\rho}\prod_{j\geq1}\bigl(2\,(2j-1)!\bigr)^{-m_j}.
\]
Here $\chi^\lambda(\rho)$ is the irreducible character of $S_n$ indexed by
$\lambda$, evaluated at cycle type $\rho$, and
$z_\rho=\prod_{j\geq1}j^{m_j}m_j!$. Thus the Schur Eisenstein series are
irreducible-character traces of the partition Eisenstein series of \cite{AGOS}.
\end{rem}

The Schur multiple zeta values of
Nakasuji--Phuksuwan--Yamasaki \cite{NPY} with all entries $2$ are exactly the
specialization of Schur functions at $x_m=1/m^2$,
\[
  \zeta\bigl(\{2\}^{\lambda}\bigr)
  = \sum_{M\in\SSYT(\lambda)}\prod_{(i,j)\in D(\lambda)}\frac{1}{m_{ij}^2}
  = s_\lambda\Bigl(\frac{1}{1^2},\frac{1}{2^2},\frac{1}{3^2},\dots\Bigr).
\]
Recall the renormalized values $\beta_{2m}$ and $\beta(\lambda)$ from the
introduction. For the single Riemann zeta value one has the closed form
$\beta_{2m}=\zeta(2m)/(2\pi i)^{2m}=-B_{2m}/(2\,(2m)!)$, and using
$p_j(1/m^2)_{m\geq1}=\zeta(2j)$ together with the homogeneity of $s_\lambda$, the Schur
value is the evaluation of $s_\lambda$ at the point $p_j=\beta_{2j}$,
\begin{equation}\label{eq:schurmzv}
  \beta(\lambda) = s_\lambda\big|_{p_j\,\mapsto\,\beta_{2j}}\in\Q.
\end{equation}
Since the constant term of $G_{2n}$ is $\beta_{2n}$ and $\PhiG(p_n)=G_{2n}$, the constant
term of $G(\lambda)$ equals $\beta(\lambda)$. For columns,
$\beta\bigl((1^r)\bigr)=\frac{(-1)^r}{4^r(2r+1)!}$ by
$\zeta(\{2\}^r)=\frac{\pi^{2r}}{(2r+1)!}$.

The same evaluation extends to skew shapes: for $\mu\subseteq\lambda$,
\begin{equation}\label{eq:beta}
  \beta(\lambda/\mu)
  = s_{\lambda/\mu}\big|_{p_j\,\mapsto\,\beta_{2j}}
  = \Phib(s_{\lambda/\mu}),
\end{equation}
where $\Phib:\Lam\to\Q$ is the algebra character $\Phib(p_m)=\beta_{2m}$ (the exponent
$2|\lambda/\mu|$ in the definition of $\beta$ being twice the number of boxes because
every entry equals $2$). In particular $\beta(\lambda/\mu)$ depends only on the symmetric
function $s_{\lambda/\mu}$. This character reappears in the coproduct factorization
$\theta=(\Phib\otimes\bar\theta)\circ\Delta$ of \cref{subsec:natural}.

\subsection{The transition theorem}\label{subsec:transition}
The starting point is the generating series identity of \cite[Thm.~1.1]{B2}: with
$X=2\sin\bigl(\tfrac T2\bigr)$,
\begin{equation}\label{eq:AlintermsofGk}
  1+\sum_{r\geq 1}A_r(q)\,X^{2r}
  = \frac{2}{X}\arcsin\!\Bigl(\frac X2\Bigr)
    \exp\!\Bigl(\sum_{j\geq 1}\frac{(-1)^{j-1}}{j}G_{2j}(q)\bigl(2\arcsin(\tfrac X2)\bigr)^{2j}\Bigr),
\end{equation}
equivalently
\begin{equation}\label{eq:g222intermsofAl}
  \sum_{l\geq 0}G_{\{2\}^l}\,T^{2l+1}
  = T\exp\!\Bigl(\sum_{j\geq 1}\frac{(-1)^{j-1}}{j}G_{2j}(q)\,T^{2j}\Bigr)
  = \sum_{l\geq 0}A_l(q)\,\Bigl(2\sin\Bigl(\frac T2\Bigr)\Bigr)^{2l+1}.
\end{equation}
For example,
\[
  G_2 = A_1-\tfrac1{24},\quad
  G_{2,2}=A_2-\tfrac18 A_1+\tfrac1{1920},\quad
  G_{2,2,2}=A_3-\tfrac{5}{24}A_2+\tfrac{13}{1920}A_1-\tfrac1{322560}.
\]
The identity \eqref{eq:g222intermsofAl} compares the images of the $e_l$ under $\PhiG$
and $\Phig$. To promote it to all shapes we recall the central factorial numbers (see
e.g.~\cite{BSSV}). Let
\[
  x^{[2m]}:=x^2\prod_{i=1}^{m-1}(x^2-i^2),
\]
and define integers $t(2m,2k)$ and $T(2m,2j)$ by
\begin{equation}\label{eq:centralfactorial}
  x^{[2m]}=\sum_{k=1}^{m}t(2m,2k)\,x^{2k},
  \qquad
  x^{2m}=\sum_{j=1}^{m}T(2m,2j)\,x^{[2j]},
\end{equation}
the central factorial numbers of the first and second kind. They satisfy
$T(2m,2j)=T(2m-2,2j-2)+j^2\,T(2m-2,2j)$, the first values being
\[
  \bigl(T(2m,2j)\bigr)_{1\leq j\leq m}
  = (1),\ (1,1),\ (1,5,1),\ (1,21,14,1),\dots
  \quad(m=1,2,3,4),
\]
and the two triangular matrices in \eqref{eq:centralfactorial} are inverse to each
other.

\begin{thm}\label{thm:transition}
Let $\theta:\Lam\to\Lam$ be the $\Q$-algebra homomorphism determined by
\begin{equation}\label{eq:thetap}
  \theta(p_m)=\beta_{2m}+\sum_{j=1}^{m}\frac{m}{j}\,\frac{(2j)!}{(2m)!}\,T(2m,2j)\,p_j
  \qquad(m\geq1).
\end{equation}
\begin{enumerate}
\item One has $\PhiG=\Phig\circ\theta$. Equivalently, for every skew shape,
\[
  G(\lambda/\mu)=\Phig\bigl(\theta(s_{\lambda/\mu})\bigr).
\]
\item $\theta$ is an automorphism of $\Lam$, with inverse given by central factorial
numbers of the first kind,
\[
  \theta^{-1}(p_m)=\sum_{k=1}^{m}\frac{(2k-1)!}{(2m-1)!}\,t(2m,2k)\,
  \bigl(p_k-\beta_{2k}\bigr).
\]
\item Write $\theta(s_\lambda)=\sum_{\mu}d_\lambda(\mu)\,s_\mu$ with $d_\lambda(\mu)\in\Q$
(a finite sum, including $\mu=\varnothing$). Then
\[
  G(\lambda)=\sum_{\mu}d_\lambda(\mu)\,g(\mu)
  \qquad\text{with}\qquad
  d_\lambda(\lambda)=1,
\]
and $d_\lambda(\mu)=0$ unless $\mu=\lambda$ or $|\mu|<|\lambda|$. Moreover
$d_\lambda(\mu)=0$ unless $\mu\subseteq\lambda$. In particular $\theta$ and
$\theta^{-1}$ preserve the width filtration \eqref{eq:widthfil}.
\item $d_\lambda(\varnothing)=\beta(\lambda)$, the renormalized Schur multiple
zeta value \eqref{eq:schurmzv}.
\item $G(\lambda)$ is the homogeneous component of weight $2|\lambda|$ of $g(\lambda)$.
In particular the weight-graded pieces of $\M$ are spanned by Schur Eisenstein series,
$\M_{2n}=\langle G(\lambda)\mid \lambda\vdash n\rangle_\Q$. Writing
$F_{\leq k}\M:=\bigoplus_{j\leq k}\M_j$, one also has
$F_{\leq 2n}\M=\langle g(\lambda)\mid |\lambda|\leq n\rangle_\Q$.
\end{enumerate}
\end{thm}

\begin{proof}
Let $u:=2\sin\bigl(\tfrac T2\bigr)=T-\frac{T^3}{24}+\frac{T^5}{1920}+\cdots$ and let $\theta_0:\Lam\to\Lam$ be
the algebra homomorphism defined by the coefficient extraction
\begin{equation}\label{eq:thetaE}
  \sum_{l\geq0}\theta_0(e_l)\,T^{2l+1} := \sum_{l\geq0}e_l\,u^{2l+1},
\end{equation}
which makes sense because the right-hand side is an odd power series in $T$ whose
coefficient of $T^{2l+1}$ only involves $e_0,\dots,e_l$. Applying $\Phig$ to
\eqref{eq:thetaE} and comparing with \eqref{eq:g222intermsofAl} gives
\begin{equation}
    \sum_{l\geq0}\Phig(\theta_0(e_l))T^{2l+1}=\sum_{l\geq0}A_l u^{2l+1}=\sum_{l\geq0}G_{\{2\}^l}T^{2l+1},
\end{equation}
i.e., $\Phig(\theta_0(e_l))=G_{\{2\}^l}=\PhiG(e_l)$, hence $\PhiG=\Phig\circ\theta_0$ on the
generators $e_l$ and therefore on all of $\Lam$.

We next show that $\theta_0=\theta$ by computing the image of each power sum. Dividing \eqref{eq:thetaE} by $T$ gives
\begin{equation}\label{eq:thetaE-factorized}
  \theta_0\bigl(E(T^2)\bigr) = \frac{u}{T}E(u^2),  \qquad u=2\sin\Bigl(\frac{T}{2}\Bigr).
\end{equation}
Both sides of \eqref{eq:thetaE-factorized} belong to $1+T^2\Lam[[T]]$, so their formal logarithms are defined, and since $\theta_0$ is an algebra homomorphism, \eqref{eq:newton} gives
\begin{equation}\label{eq:theta0-powersum-series}
  \sum_{j\geq1} \frac{(-1)^{j-1}}{j}\theta_0(p_j)T^{2j}  = \log\Bigl(\frac{u}{T}\Bigr) + \sum_{j\geq1} \frac{(-1)^{j-1}}{j}p_j u^{2j},
\end{equation}
where all power series are interpreted formally at the origin. It remains to expand $u^{2j}$ and $\log(u/T)$ as power series in $T$.

For the expansion of $u^{2j}$, define for $j\geq1$ a linear functional on $\Q[x]$ by
\begin{equation}\label{eq:central-functional}
  \mathcal L_j(f)  :=  \sum_{r=0}^{2j} (-1)^r\binom{2j}{r}f(r-j).
\end{equation}
In terms of the difference operator $(\nabla f)(x):=f(x+1)-f(x)$ one has $\mathcal L_j(f) = \bigl(\nabla^{2j}f\bigr)(-j)$. Since each application of $\nabla$ lowers the degree of a nonconstant polynomial by one and multiplies its leading coefficient by its degree, it follows that $\mathcal L_j(f)=0$ when $\deg f<2j$, while $\mathcal L_j(f)=(2j)!$ when $f$ is monic of degree $2j$. Recalling $x^{[2k]} = x^2\prod_{s=1}^{k-1}(x^2-s^2)$, we claim that
\begin{equation}\label{eq:central-functional-basis}
  \mathcal L_j\bigl(x^{[2k]}\bigr)  = \begin{cases}
    (2j)!,&k=j,\\
    0,&k\neq j.
  \end{cases}
\end{equation}
For $k<j$ this holds because $\deg x^{[2k]}=2k<2j$, and for $k=j$ because $x^{[2j]}$ is monic of degree $2j$. For $k>j$, each integer $r-j$ with $0\leq r\leq2j$ belongs to $\{-j,\ldots,j\}$, and since $j\leq k-1$ all these integers are zeros of $x^{[2k]}$, so every summand in \eqref{eq:central-functional} vanishes. This proves \eqref{eq:central-functional-basis}. Applying $\mathcal L_j$ to the second identity in \eqref{eq:centralfactorial}, we obtain
\begin{equation}\label{eq:central-functional-power}
  \mathcal L_j(x^{2m})  =  (2j)!\,T(2m,2j) \qquad  (m\geq j).
\end{equation}
For $m<j$ one has $\mathcal L_j(x^{2m})=0$ because $2m<2j$, and $\mathcal L_j(x^n)=0$ for odd $n$ because the summands indexed by $r$ and $2j-r$ in \eqref{eq:central-functional} then cancel. We now apply $\mathcal L_j$ coefficientwise to the exponential series $e^{xz}$. On the one hand,
\[
  \mathcal L_j(e^{xz}) = \sum_{r=0}^{2j}  (-1)^r\binom{2j}{r}e^{(r-j)z} = e^{-jz}(1-e^z)^{2j} =
  \bigl(e^{z/2}-e^{-z/2}\bigr)^{2j} =
  \Bigl(2\sinh\Bigl(\frac{z}{2}\Bigr)\Bigr)^{2j},
\]
while on the other hand \eqref{eq:central-functional-power} gives
\[
  \mathcal L_j(e^{xz}) = \sum_{n\geq0}\mathcal L_j(x^n)\frac{z^n}{n!} = (2j)!\sum_{m\geq j} T(2m,2j)\frac{z^{2m}}{(2m)!}.
\]
Therefore,
\begin{equation}\label{eq:centralfactorial-egf}
  \Bigl(2\sinh\Bigl(\frac{z}{2}\Bigr)\Bigr)^{2j} = (2j)!\sum_{m\geq j} T(2m,2j)\frac{z^{2m}}{(2m)!}.
\end{equation}
Substituting $z=iT$, so that $2\sinh(z/2)$ becomes $2i\sin(T/2)$, and multiplying both sides by $(-1)^j$, we obtain
\begin{equation}\label{eq:sine-centralfactorial}
  \Bigl(2\sin\Bigl(\frac{T}{2}\Bigr)\Bigr)^{2j}  =  (2j)!\sum_{m\geq j}  (-1)^{m-j}\frac{T(2m,2j)}{(2m)!}T^{2m}.
\end{equation}

For the logarithmic term in \eqref{eq:theta0-powersum-series}, recall that the Bernoulli numbers are defined by $\frac{z}{e^z-1} = \sum_{n\geq0}B_n\frac{z^n}{n!}$, so that $B_1=-1/2$ and $B_{2m+1}=0$ for every $m\geq1$. Combining the identity $\cot z =  i\,\frac{e^{2iz}+1}{e^{2iz}-1}  =  i+\frac{2i}{e^{2iz}-1}$ with this generating series gives
\[
  \cot z = i+\frac{1}{z} \sum_{n\geq0}B_n\frac{(2iz)^n}{n!} = \frac{1}{z} + \sum_{m\geq1} (-1)^m \frac{2^{2m}B_{2m}}{(2m)!}z^{2m-1},
\]
where the term containing $B_1=-1/2$ cancels the initial term $i$ and the odd Bernoulli numbers of higher index vanish. Substituting $z=T/2$ and $\beta_{2m}= -\frac{B_{2m}}{2(2m)!}$ gives
\begin{equation}\label{eq:cot-beta}
  \frac{1}{2}\cot\Bigl(\frac{T}{2}\Bigr)-\frac{1}{T} =  \sum_{m\geq1} (-1)^m\frac{B_{2m}}{(2m)!}T^{2m-1} =  2\sum_{m\geq1}  (-1)^{m-1}\beta_{2m}T^{2m-1}.
\end{equation}
Since $u=2\sin(T/2)$, the left-hand side of \eqref{eq:cot-beta} equals $\frac{d}{dT}\log\bigl(\frac{u}{T}\bigr)$, and as $u/T$ has constant term $1$, the logarithm $\log(u/T)$ has constant term $0$. Termwise integration of \eqref{eq:cot-beta} therefore gives
\begin{equation}\label{eq:log-sine-beta}
  \log\Bigl(\frac{2\sin(T/2)}{T}\Bigr)  = \sum_{m\geq1}  \frac{(-1)^{m-1}}{m}\beta_{2m}T^{2m}.
\end{equation}

Substituting \eqref{eq:sine-centralfactorial} and \eqref{eq:log-sine-beta} into \eqref{eq:theta0-powersum-series}, the coefficient of $T^{2m}$ on the right-hand side becomes, for $m\geq1$,
\[
  \frac{(-1)^{m-1}}{m}\beta_{2m}+\sum_{j=1}^{m}\frac{(-1)^{j-1}}{j}p_j(2j)!(-1)^{m-j} \frac{T(2m,2j)}{(2m)!}= (-1)^{m-1}\left( \frac{\beta_{2m}}{m} + \sum_{j=1}^{m}\frac{(2j)!}{j(2m)!}T(2m,2j)p_j\right),
\]
while on the left-hand side it is $\frac{(-1)^{m-1}}{m}\theta_0(p_m)$. Comparing the two and multiplying by $m(-1)^{m-1}$ gives
\[
  \theta_0(p_m) =  \beta_{2m}  + \sum_{j=1}^{m} \frac{m}{j}\frac{(2j)!}{(2m)!}  T(2m,2j)p_j,
\]
which is precisely \eqref{eq:thetap}. Hence $\theta_0=\theta$, and together with $\PhiG=\Phig\circ\theta_0$ this proves part~\textup{(i)}.

For \textup{(ii)}, note that \eqref{eq:thetap} is unitriangular with respect to the filtration
by $m$, so $\theta$ is bijective. Let $\eta$ be the algebra homomorphism defined by the
displayed formula for $\theta^{-1}$, i.e.,
\begin{equation}\label{eq:etapm}
    \eta(p_m)=\sum_{k=1}^m\frac{(2k-1)!}{(2m-1)!}t(2m,2k)(p_k-\beta_{2k}).
\end{equation}
 Substituting
$$\theta(p_k)-\beta_{2k}=\sum_{j}\frac kj\frac{(2j)!}{(2k)!}T(2k,2j)\,p_j$$ into \eqref{eq:etapm} we have
\[
  \theta\bigl(\eta(p_m)\bigr)
  =\sum_{j=1}^{m}\frac{(2j)!}{2j\,(2m-1)!}
    \Bigl(\sum_{k=j}^{m}t(2m,2k)\,T(2k,2j)\Bigr)p_j
  = p_m,
\]
since the matrices in \eqref{eq:centralfactorial} are mutually inverse, i.e.\ $\sum_{k=j}^mt(2m,2k)T(2k,2j)=\delta_{m,j}$, only the term $j=m$ survives, and its coefficient is
$\frac{(2m)!}{2m\,(2m-1)!}=1$. Hence $\eta=\theta^{-1}$.

For (iii), the first two statements follow since \eqref{eq:thetap} gives
$\theta(s_\lambda)=s_\lambda+(\text{terms of degree}<|\lambda|)$, and applying $\Phig$
turns the Schur expansion of $\theta(s_\lambda)$ into the $g$-expansion. The support
statement ``$d_\lambda(\mu)=0$ unless $\mu\subseteq\lambda$" is \cref{cor:winding-support}, whose proof in \cref{subsec:natural} uses
only part~(i) and \cref{cor:rowcol}. The statement about
the width filtration follows because $\Filw_n\Lam$ is spanned by the $s_\mu$ with
$\mu_1\leq n$.

(iv) is the evaluation of $\theta(s_\lambda)$ at $p_j=0$ for all $j$, which by
\eqref{eq:thetap} equals the evaluation of $s_\lambda$ at $p_j=\beta_{2j}$, and
this is $\beta(\lambda)$ by \eqref{eq:schurmzv}.

(v) By the proof of \cref{prop:RgInMz}, each $P_r=\Phig(p_r)$ is a $\Q$-linear
combination of $1,G_2,\dots,G_{2r}$, so $\Phig$ maps symmetric functions of
degree $d$ into $F_{\leq2d}\M$. In particular $g(\mu)\in F_{\leq2|\mu|}\M$. By
(iii), $G(\lambda)-g(\lambda)$ is a $\Q$-linear combination of $g(\mu)$ with
$|\mu|<|\lambda|$, hence lies in $F_{\leq2|\lambda|-2}\M$, while $G(\lambda)$ is
homogeneous of weight $2|\lambda|$ by \cref{thm:jt}. This proves the first
claim. Since $\M_{2n}$ is spanned by the monomials in the $G_{2j}$ of weight
$2n$, i.e.\ by the $\PhiG(p_\nu)$ with $\nu\vdash n$, it is also spanned by the
$G(\lambda)$ with $\lambda\vdash n$. Conversely, expressing each $G(\lambda)$
with $\lambda\vdash k\leq n$ through (iii) shows
$\M_{2k}\subseteq\langle g(\mu)\mid|\mu|\leq k\rangle_\Q$, and together with
$g(\mu)\in F_{\leq2|\mu|}\M$ this gives
$F_{\leq2n}\M=\langle g(\lambda)\mid|\lambda|\leq n\rangle_\Q$.
\end{proof}

\begin{ex}\label{ex:transitiontable}
The first values of $\theta$ on power sums are
\[
  \theta(p_1)=p_1-\tfrac1{24},\quad
  \theta(p_2)=p_2+\tfrac16p_1+\tfrac1{1440},\quad
  \theta(p_3)=p_3+\tfrac14p_2+\tfrac1{120}p_1-\tfrac1{60480},
\]
\[
  \theta(p_4)=p_4+\tfrac13p_3+\tfrac1{40}p_2+\tfrac1{5040}p_1+\tfrac1{2419200},
\]
and the resulting expansions of the Schur Eisenstein series with at most three boxes in
Schur MacMahon series read, writing the shapes as Young diagrams,
\begin{align}
  \pes{1} &= -\tfrac{1}{24}+\pgs{1},\\[4pt]
  \pes{2} &= \tfrac{7}{5760}+\tfrac1{24}\,\pgs{1}+\pgs{2},\\[4pt]
  \pes{1,1} &= \tfrac{1}{1920}-\tfrac18\,\pgs{1}+\pgs{1,1},\\[4pt]
  \pes{3} &= -\tfrac{31}{967680}+\tfrac1{1920}\,\pgs{1}+\tfrac18\,\pgs{2}+\pgs{3},\\[4pt]
  \pes{2,1} &= -\tfrac{1}{53760}-\tfrac1{960}\,\pgs{1}-\tfrac18\,\pgs{2}
                        +\tfrac1{24}\,\pgs{1,1}+\pgs{2,1},\\[4pt]
  \pes{1,1,1} &= -\tfrac{1}{322560}+\tfrac{13}{1920}\,\pgs{1}
                        -\tfrac5{24}\,\pgs{1,1}+\pgs{1,1,1}.
\end{align}
The constant terms are the renormalized Schur multiple zeta values, e.g.\
$\beta\bigl((2,1)\bigr)=-\tfrac{1}{53760}$.
\end{ex}

\begin{cor}\label{cor:rowcol}
In addition to the column identity \eqref{eq:g222intermsofAl} one has the row identity
\[
  2\sinh\Bigl(\frac T2\Bigr)\sum_{l\geq 0}G\bigl((l)\bigr)\,T^{2l}
  = T\,\sum_{l\geq 0}g((l))\,\Bigl(2\sinh\Bigl(\frac T2\Bigr)\Bigr)^{2l}.
\]
In particular the single-row series $G((l))$ are $\Q$-linear combinations of the
single-row series $g((m))=h_m(x_1,x_2,\dots)$ with $m\leq l$, and conversely.
\end{cor}

\begin{proof}
Apply the substitution $T\mapsto iT$ to \eqref{eq:thetaE} and use
$2\sin(iT/2)=2i\sinh(T/2)$ together with $H(t)=E(-t)^{-1}$ from \eqref{eq:newton}. This
gives $\sum_l\theta(h_l)v\,T^{2l}=T\sum_l h_l\,v^{2l}$ with $v=2\sinh(T/2)$, and
applying $\Phig$ yields the claim since $\PhiG(h_l)=G((l))$ and $\Phig(h_l)=g((l))$.
\end{proof}

The transition coefficients of a fixed weight difference are, moreover, \emph{integer}
combinations of renormalized Schur multiple zeta values.

\begin{prop}\label{prop:zbeta}
For every partition $\mu$ with $|\mu|\leq|\lambda|$ the coefficient $d_\lambda(\mu)$ is an integer linear
combination of the renormalized Schur multiple zeta values of weight
$2(|\lambda|-|\mu|)$,
\[
  d_\lambda(\mu)=\sum_{\nu\,\vdash\,|\lambda|-|\mu|}a_{\lambda\mu}(\nu)\,\beta(\nu),
  \qquad a_{\lambda\mu}(\nu)\in\Z .
\]
More precisely $d_\lambda(\mu)=\Phib(F_{\lambda\mu})$ for an integral homogeneous symmetric
function $F_{\lambda\mu}\in\Lam_\Z$ of degree $|\lambda|-|\mu|$, with
$a_{\lambda\mu}(\nu)=\langle F_{\lambda\mu},s_\nu\rangle$.
\end{prop}

\begin{proof}
Recall from \eqref{eq:beta} that $\beta(\nu)=\Phib(s_\nu)$ for the character
$\Phib(p_m)=\beta_{2m}$. Since $\Phib(e_a)=\beta\bigl((1^a)\bigr)=\frac{(-1)^a}{4^a(2a+1)!}$,
the elementary generating series \eqref{eq:newton} specializes under $\Phib$ to the column
series $u=2\sin(\tfrac T2)$,
\begin{equation}\label{eq:chiE}
  \sum_{a\geq0}\Phib(e_a)\,T^{2a+1}=2\sin\bigl(\tfrac T2\bigr)=u,
  \qquad\text{i.e.}\quad \Phib\bigl(E(T^2)\bigr)=u/T .
\end{equation}
Define $$\Xi^{(N)}_m:=[t^m]E(t)^N=\sum_{a_1+\dots+a_N=m}e_{a_1}\cdots e_{a_N},$$ an integral
homogeneous symmetric function of degree $m$. Extracting the coefficient of $T^{2l+1}$
from \eqref{eq:thetaE} (recall $\theta_0=\theta$), $$\sum_{l\geq0}\theta(e_l)T^{2l+1}=\sum_{l\geq0} e_l\,u^{2l+1},$$
and using that $\Phib$ is multiplicative together with \eqref{eq:chiE} gives
\begin{equation}\label{eq:thetae-lift}
  \theta(e_l)=\sum_{k=0}^{l}\Phib\bigl(\Xi^{(2k+1)}_{l-k}\bigr)\,e_k
  \qquad(l\geq0).
\end{equation}
Let $\Delta_{\!F}:\Lam\to\Lam\otimes\Lam$ be the algebra homomorphism with
\[
  \Delta_{\!F}(e_l)=\sum_{k=0}^{l}\Xi^{(2k+1)}_{l-k}\otimes e_k .
\]
This is the co-opposite of the Fa\`a di Bruno deformation $FdB_\gamma$ of
Foissy \cite{Fo} at $\gamma=-2$.
For the equivalent alphabet form, let $X=(x_1,x_2,\ldots)$ and
$Y=(y_1,y_2,\ldots)$ be independent alphabets, identify
$f\otimes g$ with $f(X)g(Y)$, and write
$E_X(t):=\sum_{l\geq0}e_l(X)t^l=\prod_i(1+x_it)$, with $E_Y$ defined similarly.
Then, in generating-series form,
\begin{equation}\label{eq:fdbcoprodE}
  \Delta_{\!F}E(t)=\sum_{k\geq0}t^kE(t)^{2k+1}\otimes e_k,
  \qquad
  \bigl(\Delta_{\!F}E(t)\bigr)(X,Y)=E_X(t)E_Y\bigl(tE_X(t)^2\bigr).
\end{equation}
Here the second factor means formal composition: first set
$u=tE_X(t)^2$, and then substitute $u$ into
$E_Y(u)=\sum_{k\geq0}e_k(Y)u^k$. Thus the right-hand side expands as
$\sum_{k\geq0}t^kE_X(t)^{2k+1}e_k(Y)$, which is exactly the first formula
under the above identification.
Each summand lies in $\Lam_\Z\otimes\Lam_\Z$ and has total degree $(l-k)+k=l$, so
$\Delta_{\!F}$ preserves the total degree and maps $\Lam_\Z$ into
$\Lam_\Z\otimes\Lam_\Z$. By
\eqref{eq:thetae-lift} the algebra homomorphisms
$(\Phib\otimes\mathrm{id})\circ\Delta_{\!F}$ and
$\theta$ agree on the generators $e_l$, hence coincide. Expanding in the second tensor
factor,
\begin{equation}\label{eq:fdbSchurconstants}
  \Delta_{\!F}(s_\lambda)=\sum_\mu F_{\lambda\mu}\otimes s_\mu,
\end{equation}
and by the integrality and homogeneity of $\Delta_{\!F}$ each
$F_{\lambda\mu}\in\Lam_\Z$ is homogeneous of degree
$|\lambda|-|\mu|$. Pairing with $s_\mu$,
\[
  d_\lambda(\mu)=\langle\theta(s_\lambda),s_\mu\rangle
  =\bigl\langle(\Phib\otimes\mathrm{id})\Delta_{\!F}(s_\lambda),s_\mu\bigr\rangle
  =\Phib(F_{\lambda\mu})
  =\sum_{\nu}\langle F_{\lambda\mu},s_\nu\rangle\,\beta(\nu),
\]
with $\langle F_{\lambda\mu},s_\nu\rangle\in\Z$ because $F_{\lambda\mu}\in\Lam_\Z$.
\end{proof}

\subsection{The Fa\`a di Bruno coproduct}\label{subsec:fdb}
For classical multiple Eisenstein series the Fourier expansion is, by
the first author and Tasaka \cite{BT}, a convolution of multiple zeta values with the
generating series of divisor sums with respect to the Goncharov coproduct. We now
prove the exact analogue for the two Schur families: there is an explicit coproduct
on $\Lam$ for which $G$ is the convolution of $\beta$ and $g$.

Recall the integral graded algebra map $\Delta_{\!F}$ introduced in
\eqref{eq:fdbcoprodE}. In complete homogeneous coordinates it has the equivalent form
\begin{equation}\label{eq:fdbcoprod}
  \Delta_{\!F}\,H(t)=\sum_{k\geq0}t^k\,H(t)^{1-2k}\otimes h_k,
  \qquad H(t)=\sum_{l\geq0}h_l\,t^l,
\end{equation}
that is, $\Delta_{\!F}(h_n)=\sum_{k=0}^{n}\bigl([t^{\,n-k}]H(t)^{1-2k}\bigr)\otimes h_k$.

\begin{ex}\label{ex:fdb21}
Taking coefficients in \eqref{eq:fdbcoprodE} gives
\[
\begin{aligned}
  \Delta_{\!F}(e_1)
    &=e_1\otimes1+1\otimes e_1,\\
  \Delta_{\!F}(e_2)
    &=e_2\otimes1+3e_1\otimes e_1+1\otimes e_2,\\
  \Delta_{\!F}(e_3)
    &=e_3\otimes1+3(e_1^2+e_2)\otimes e_1
      +5e_1\otimes e_2+1\otimes e_3.
\end{aligned}
\]
The Jacobi--Trudi formula gives $s_{(2,1)}=e_1e_2-e_3$. Since
$\Delta_{\!F}$ is an algebra homomorphism, the preceding identities yield
\[
\begin{aligned}
  \Delta_{\!F}\bigl(s_{(2,1)}\bigr)
  &=\Delta_{\!F}(e_1)\Delta_{\!F}(e_2)-\Delta_{\!F}(e_3)\\
  &=s_{(2,1)}\otimes1-2e_2\otimes e_1
    +e_1\otimes(3e_1^2-4e_2)+1\otimes s_{(2,1)}\\
  &=s_{(2,1)}\otimes1-2s_{(1,1)}\otimes s_{(1)}
    +3s_{(1)}\otimes s_{(2)}-s_{(1)}\otimes s_{(1,1)}
    +1\otimes s_{(2,1)},
\end{aligned}
\]
where the last equality uses $e_1=s_{(1)}$, $e_2=s_{(1,1)}$, and
$e_1^2=s_{(2)}+s_{(1,1)}$. Applying
$m\circ(\Phib\otimes\Phig)$ gives exactly \eqref{eq:introG21}.
\end{ex}

\begin{lem}\label{lem:fdb}
\begin{enumerate}
    \item $(\Lam_\Z,\,\cdot\,,\Delta_{\!F})$ is a connected graded Hopf algebra with counit
$\varepsilon(h_n)=\delta_{n,0}$. In particular the same is true after base change to
$\Q$.
\item 
For two characters $\varphi,\psi:\Lam\to A$ with values in a commutative
$\Q$-algebra $A$, write $H_\varphi(t):=\sum_n\varphi(h_n)t^n$. Then the convolution
character $\varphi\star\psi:=m_A\circ(\varphi\otimes\psi)\circ\Delta_{\!F}$ satisfies
\[
  H_{\varphi\star\psi}(t)=H_\varphi(t)\,H_\psi\bigl(U_\varphi(t)\bigr),
  \qquad U_\varphi(t):=\frac{t}{H_\varphi(t)^2},
\]
and consequently $U_{\varphi\star\psi}=U_\psi\circ U_\varphi$: the map
$\varphi\mapsto U_\varphi$ reverses convolution into composition of formal
diffeomorphisms.
\end{enumerate}

\end{lem}

\begin{proof}
We first prove that $\Delta_{\!F}$ is an algebra homomorphism respecting the grading of
\[
  \Lam_{\Z} = \Z[e_1,e_2,\ldots] = \Z[h_1,h_2,\ldots]  = \bigoplus_{n\geq0}\Lam_{\Z,n},
\]
where $\deg(e_n)=\deg(h_n)=n$ and $\Lam_{\Z,0}=\Z$. Taking the coefficient of $t^l$ in \eqref{eq:fdbcoprodE} gives
\[
  \Delta_{\!F}(e_l) = \sum_{k=0}^{l} \left([t^{\,l-k}]E(t)^{2k+1}\right)\otimes e_k,
\]
and since the $e_l$ are algebraically independent, these formulas define a unique unital algebra homomorphism $\Delta_{\!F}: \Lam_{\Z} \to \Lam_{\Z}\otimes_{\Z}\Lam_{\Z}$. As the coefficient $[t^{\,l-k}]E(t)^{2k+1}$ is an integral symmetric function of degree $l-k$, we obtain
\[
  \Delta_{\!F}(\Lam_{\Z,n}) \subseteq \bigoplus_{a+b=n} \Lam_{\Z,a}\otimes_{\Z}\Lam_{\Z,b},
\]
so $\Delta_{\!F}$ is integral and preserves the total grading.

To derive \eqref{eq:fdbcoprod}, write $E_1(t)=E(t)\otimes1$ and $E_2(t)=1\otimes E(t)$ in two tensor factors, with the same notation for $H(t)$, so that \eqref{eq:fdbcoprodE} reads $\Delta_{\!F}E(t) = E_1(t)E_2\bigl(tE_1(t)^2\bigr)$. Since $H(t)=E(-t)^{-1}$, we obtain
\[
  \Delta_{\!F}H(t) = \bigl(\Delta_{\!F}E(-t)\bigr)^{-1} = \Bigl( E_1(-t) E_2\bigl(-tE_1(-t)^2\bigr) \Bigr)^{-1} = H_1(t) H_2\bigl(tH_1(t)^{-2}\bigr),
\]
and expanding the second factor gives
\[
  \Delta_{\!F}H(t) = H_1(t) \sum_{k\geq0} h_k^{(2)} \bigl(tH_1(t)^{-2}\bigr)^k =  \sum_{k\geq0} t^k H_1(t)^{1-2k}h_k^{(2)},
\]
which is \eqref{eq:fdbcoprod}. All these identities are understood coefficientwise, so every coefficient belongs to the ordinary tensor product.

For coassociativity, let $H_i(t)$ denote $H(t)$ in the $i$-th tensor factor of $\Lam_{\Z}^{\otimes3}[[t]]$ and put $U_i(s):=s/H_i(s)^2$, so that the preceding identity takes the form $\Delta_{\!F}H(t)  = H_1(t)H_2\bigl(U_1(t)\bigr)$. Applying $\Delta_{\!F}$ to the first tensor factor gives
\[
\begin{aligned}
  \bigl((\Delta_{\!F}\otimes\mathrm{id}) \circ\Delta_{\!F}\bigr)H(t) &= H_1(t)H_2\bigl(U_1(t)\bigr) H_3\left( \frac{t}{H_1(t)^2H_2(U_1(t))^2}\right)\\
  &= H_1(t)H_2\bigl(U_1(t)\bigr) H_3\bigl(U_2(U_1(t))\bigr),
\end{aligned}
\]
while applying it to the second tensor factor gives
\[
\begin{aligned}
  \bigl((\mathrm{id}\otimes\Delta_{\!F}) \circ\Delta_{\!F}\bigr)H(t) &=  H_1(t)H_2\bigl(U_1(t)\bigr) H_3\left( \frac{U_1(t)}{ H_2(U_1(t))^2 } \right)\\
  &= H_1(t)H_2\bigl(U_1(t)\bigr) H_3\bigl(U_2(U_1(t))\bigr).
\end{aligned}
\]
Comparing the coefficients of $t^n$ in these two equal expressions gives $(\Delta_{\!F}\otimes\mathrm{id}) \Delta_{\!F}(h_n) = (\mathrm{id}\otimes\Delta_{\!F}) \Delta_{\!F}(h_n)$ for every $n\geq0$, and since the $h_n$ generate $\Lam_{\Z}$ and both sides are algebra homomorphisms, $\Delta_{\!F}$ is coassociative on all of $\Lam_{\Z}$.

The counit is $\varepsilon:\Lam_{\Z}\to\Z$, $\varepsilon(h_n)=\delta_{n,0}$, which defines a unital algebra homomorphism since $\Lam_{\Z}=\Z[h_1,h_2,\ldots]$ and which satisfies $\varepsilon(H(t))=1$. Applying $\varepsilon$ to the first tensor factor of \eqref{eq:fdbcoprod} gives
\[
  (\varepsilon\otimes\mathrm{id}) \Delta_{\!F}H(t) = \sum_{k\geq0} t^k \varepsilon\bigl(H(t)^{1-2k}\bigr)h_k = \sum_{k\geq0}h_kt^k = H(t),
\]
and applying it to the second tensor factor, where only the term $k=0$ remains,
\[
  (\mathrm{id}\otimes\varepsilon) \Delta_{\!F}H(t) = \sum_{k\geq0} t^kH(t)^{1-2k}\varepsilon(h_k) = H(t).
\]
Comparing coefficients proves $(\varepsilon\otimes\mathrm{id})\Delta_{\!F} = \mathrm{id} = (\mathrm{id}\otimes\varepsilon)\Delta_{\!F}$, so $\Lam_{\Z}$ is a coalgebra. The algebra and coalgebra structures are compatible because both $\Delta_{\!F}$ and $\varepsilon$ are unital algebra homomorphisms. Therefore $(\Lam_{\Z},\cdot,\Delta_{\!F},\varepsilon)$ is a graded bialgebra, and it is connected because $\Lam_{\Z,0}=\Z$.

For completeness, we construct the antipode. For $f\in\Lam_{\Z,n}$ with $n>0$, the grading and the counit identities give
\[
  \Delta_{\!F}(f) = f\otimes1+1\otimes f + \sum_{(f)} f_{(1)}\otimes f_{(2)}
\]
with $0<\deg(f_{(1)})<n$ and $0<\deg(f_{(2)})<n$. Define $S_L,S_R:\Lam_{\Z}\to\Lam_{\Z}$ by $S_L(1)=S_R(1)=1$ and
\[
  S_L(f):= -f-\sum_{(f)} S_L(f_{(1)})f_{(2)},\qquad S_R(f):=-f-\sum_{(f)} f_{(1)}S_R(f_{(2)}),
\]
which is well defined by induction on the degree because the recursive arguments $f_{(1)}$ and $f_{(2)}$ on the right have degree less than $n$. These formulas give
\[
  m_{\Lam}\circ(S_L\otimes\mathrm{id})\circ\Delta_{\!F} = u_{\Lam}\circ\varepsilon
  \qquad\text{and}\qquad
  m_{\Lam}\circ(\mathrm{id}\otimes S_R)\circ\Delta_{\!F}=u_{\Lam}\circ\varepsilon,
\]
where $u_{\Lam}:\Z\to\Lam_{\Z}$ is the unit map. The convolution product on $\operatorname{End}_{\Z}(\Lam_{\Z})$ is associative because $\Delta_{\!F}$ is coassociative, and therefore
\[
  S_L=S_L\ast(u_{\Lam}\circ\varepsilon)=S_L\ast(\mathrm{id}\ast S_R)=(S_L\ast\mathrm{id})\ast S_R=(u_{\Lam}\circ\varepsilon)\ast S_R = S_R.
\]
Their common value $S_F$ is the antipode, and since the recursive formulas use only addition, subtraction, and multiplication, it is defined over $\Z$. This proves that $\Lam_{\Z}$ is a connected graded Hopf algebra, and base change from $\Z$ to $\Q$ gives the corresponding Hopf algebra over $\Q$.

Let now $\varphi,\psi:\Lam\to A$ be characters. Since $\Delta_{\!F}$ is an algebra homomorphism and $A$ is commutative, $\varphi\star\psi$ is again a character, and applying $\varphi\otimes\psi$ to \eqref{eq:fdbcoprod} gives
\[
  H_{\varphi\star\psi}(t) = \sum_{k\geq0} t^kH_\varphi(t)^{1-2k}\psi(h_k) = H_\varphi(t) \sum_{k\geq0} \psi(h_k) \left(\frac{t}{H_\varphi(t)^2}\right)^k = H_\varphi(t) H_\psi\bigl(U_\varphi(t)\bigr).
\]
Since $H_\varphi(0)=1$, one has $U_\varphi(t) =t/H_\varphi(t)^2 \in t+t^2A[[t]]$, and finally
\[
  U_{\varphi\star\psi}(t) =\frac{t}{ H_\varphi(t)^2 H_\psi(U_\varphi(t))^2  } =\frac{U_\varphi(t)}{ H_\psi(U_\varphi(t))^2} = U_\psi\bigl(U_\varphi(t)\bigr),
\]
so $\varphi\mapsto U_\varphi$ reverses convolution into composition.
\end{proof}

\begin{thm}\label{thm:fdb}
One has
\[
  \theta=T_\beta:=(\Phib\otimes\mathrm{id})\circ\Delta_{\!F},
\]
where $\Phib:\Lam\to\Q$ is the character $\Phib(p_m)=\beta_{2m}$ of \eqref{eq:beta}.
Consequently, applying $\Phig$,
\[
  \PhiG=m\circ(\Phib\otimes\Phig)\circ\Delta_{\!F},
\]
that is, the Schur Eisenstein series are the $\Delta_{\!F}$-convolution of the
renormalized Schur multiple zeta values with the Schur MacMahon series. Moreover,
$\theta$ is the left winding automorphism associated with $\Phib$, and
\begin{equation}\label{eq:thetainversewinding}
  \theta^{-1}=T_{\beta^{\star-1}}
  =(\Phib^{\star-1}\otimes\mathrm{id})\circ\Delta_{\!F},
\end{equation}
where $\Phib^{\star-1}=\Phib\circ S_F$ is the convolution inverse.
\end{thm}

\begin{proof}
By \cref{lem:fdb}, $\Delta_{\!F}$ is a unital algebra homomorphism. Since $\Phib$ is a character, the map $T_\beta=(\Phib\otimes\mathrm{id})\circ\Delta_{\!F}$ is an algebra endomorphism of $\Lam$. The map $\theta$ is also an algebra endomorphism. Since $\Lam=\Q[h_1,h_2,\ldots]$, it is enough to compare $\theta$ and $T_\beta$ on the generators $h_n$.

Set $v:=2\sinh\left(\frac{T}{2}\right)$. Replacing $T$ by $iT$ in \eqref{eq:log-sine-beta} gives
\begin{equation}
    \log\left(\frac{v}{T}\right)=\sum_{m\geq1} \frac{(-1)^{m-1}}{m}\beta_{2m}(iT)^{2m} = -\sum_{m\geq1}\frac{\beta_{2m}}{m}\,T^{2m}.
\end{equation}
On the other hand, \eqref{eq:newton} and \eqref{eq:beta} give
\begin{equation}\label{eq:PhibH-hyperbolic}
  \Phib\bigl(H(T^2)\bigr) =\exp\left( \sum_{m\geq1} \frac{\Phib(p_m)}{m}\,T^{2m} \right) = \exp\left(  \sum_{m\geq1}  \frac{\beta_{2m}}{m}\,T^{2m}  \right)= \exp\left(  -\log\left(\frac{v}{T}\right) \right) =\frac{T}{v}.
\end{equation}

The homomorphism $\theta_0$ occurring in
\eqref{eq:thetaE-factorized} was shown to equal $\theta$ in the proof of
\cref{thm:transition}. Replacing $T$ by $iT$ in
\eqref{eq:thetaE-factorized}, and using $u(iT)=2\sin\left(\frac{iT}{2}\right)=iv$, gives
\begin{equation}
  \theta\bigl(E(-T^2)\bigr) = \frac{v}{T}E(-v^2).
\end{equation}
The series $v/T$, $E(-T^2)$, and $E(-v^2)$ have constant term $1$ and are therefore invertible. Since $H(t)=E(-t)^{-1}$ by
\eqref{eq:newton}, we obtain
\begin{equation}\label{eq:theta-row-generating}
  \sum_{n\geq0}\theta(h_n)T^{2n} = \theta\bigl(H(T^2)\bigr) = \theta\bigl(E(-T^2)\bigr)^{-1} = \left( \frac{v}{T}E(-v^2)\right)^{-1}= \frac{T}{v} \sum_{k\geq0}h_kv^{2k}.
\end{equation}

We now set $t=T^2$ in \eqref{eq:fdbcoprod} and apply
$\Phib\otimes\mathrm{id}$. Using
\eqref{eq:PhibH-hyperbolic}, we obtain
\begin{equation}\label{eq:Tbeta-row-generating}
  \sum_{n\geq0}T_\beta(h_n)T^{2n}=\sum_{n\geq0}(\Phib\otimes\mathrm{id})\Delta_{\!F}(h_n)\,T^{2n}=\sum_{k\geq0}T^{2k}\Phib\bigl(H(T^2)\bigr)^{1-2k}h_k=\frac{T}{v}\sum_{k\geq0}h_kv^{2k}.
\end{equation}
Comparing \eqref{eq:theta-row-generating} and \eqref{eq:Tbeta-row-generating} gives $T_\beta(h_n)=\theta(h_n)$ for $n\geq0$. Since the $h_n$ generate $\Lam$, it follows that
\begin{equation}
  \theta = T_\beta = (\Phib\otimes\mathrm{id})\circ\Delta_{\!F}.
\end{equation}
For every $f\in\Lam$, applying $\Phig$ to the first identity and using part~\textup{(i)} of \cref{thm:transition} gives
\begin{equation}
  \bigl( m\circ(\Phib\otimes\Phig)\circ\Delta_{\!F}\bigr)(f) =\Phig\left( (\Phib\otimes\mathrm{id}) \Delta_{\!F}(f)\right)= \Phig\bigl(\theta(f)\bigr)=\PhiG(f).
\end{equation}
Therefore
  $$\PhiG = m\circ(\Phib\otimes\Phig)\circ\Delta_{\!F}.$$
  Coassociativity gives
$T_\varphi T_\psi=T_{\psi\star\varphi}$, so the inverse of $T_\beta$ is
$T_{\beta^{\star-1}}$, proving \eqref{eq:thetainversewinding}.
\end{proof}

\subsection{Subdiagram support}\label{subsec:natural}
The transition theorem gives a canonical expansion in the free ring $\Lam$,
before the specialization $\Phig$ introduces relations:
\begin{equation}\label{eq:naturalexp}
  \theta(s_{\lambda/\mu})
  =
  \sum_\nu d_{\lambda/\mu}(\nu)s_\nu,
  \qquad
  G(\lambda/\mu)
  =
  \sum_\nu d_{\lambda/\mu}(\nu)g(\nu).
\end{equation}
We now determine the support of this expansion. The support statement holds
universally, at the level of the Fa\`a di Bruno coproduct itself, before any
character is applied. It therefore covers $\theta=T_\beta$, its inverse
$\theta^{-1}=T_{\beta^{\star-1}}$, and every other winding operator of
\cref{thm:fdb} at once.

Throughout this subsection $A$ denotes a commutative $\Q$-algebra. Write
$\Lam_A:=\Lam\otimes_\Q A$ and
\[
  \widehat{\Lam}_A
  :=
  \prod_{n\geq0}\Lam_n\otimes_\Q A
\]
for its degree completion, where $\Lam_n$ is the homogeneous component of
degree $n$. The Schur functions form an $A$-basis of $\Lam_A$, the Hall inner
product extends $A$-bilinearly, and it extends degreewise to a pairing between
$\Lam_A$ and $\widehat{\Lam}_A$.

\begin{lem}\label{lem:substitution-support}
Let $U(t)\in t+t^2A[[t]]$ and put $u_2:=[t^2]\,U(t)$. For every partition
$\nu$, the Schur expansion of
\[
  s_\nu\bigl(U(x_1),U(x_2),\ldots\bigr)\in\widehat{\Lam}_A
\]
is supported on the partitions $\rho$ satisfying $\nu\subseteq\rho$. The
coefficient of $s_\nu$ equals $1$, and for every addable box of $\nu$ in row
$r$, of content $c=\nu_r+1-r$, the coefficient of $s_{\nu+\square}$ equals
$c\,u_2$.
\end{lem}

\begin{proof}
Fix $N\geq\ell(\nu)$ and put
\[
  X=(x_1,\ldots,x_N),
  \qquad
  \delta=(N-1,N-2,\ldots,1,0),
  \qquad
  b=\nu+\delta,
\]
where zeros are added to $\nu$ if necessary. For an integer sequence
$\gamma=(\gamma_1,\ldots,\gamma_N)$, write
$a_\gamma(X):=\det\bigl(x_i^{\gamma_j}\bigr)_{1\leq i,j\leq N}$. The
bialternant formula gives $s_\nu(X)=a_b(X)/a_\delta(X)$. For $r\geq0$, write
\[
  U(t)^r
  =
  \sum_{m\geq r}u_{m,r}t^m,
\]
so that $u_{r,r}=1$ and $u_{r+1,r}=r\,u_2$. By multilinearity of the
determinant,
\begin{equation}\label{eq:alternant-substitution}
  a_b\bigl(U(X)\bigr)
  =
  \sum_{\substack{m_j\geq b_j\\1\leq j\leq N}}
  \left(\prod_{j=1}^{N}u_{m_j,b_j}\right)
  a_{(m_1,\ldots,m_N)}(X),
\end{equation}
where $U(X):=\bigl(U(x_1),\ldots,U(x_N)\bigr)$. A summand in
\eqref{eq:alternant-substitution} is zero if two of the integers
$m_1,\ldots,m_N$ are equal. Otherwise, let $c_1>c_2>\cdots>c_N\geq0$ be their
decreasing rearrangement, so that
$a_{(m_1,\ldots,m_N)}(X)=\pm a_c(X)$. For every $k$, the integers
$m_1,\ldots,m_k$ are all at least $b_k$, since
$m_j\geq b_j\geq b_k$ for $1\leq j\leq k$. Thus at least $k$ of the integers
$m_1,\ldots,m_N$ are at least $b_k$, and hence $c_k\geq b_k$ for
$1\leq k\leq N$. Set $\rho=c-\delta$. Since $c$ is strictly decreasing, $\rho$
is a partition, and $\rho_k=c_k-\delta_k\geq b_k-\delta_k=\nu_k$, so
$\nu\subseteq\rho$. Since $a_c(X)/a_\delta(X)=s_\rho(X)$, the Schur expansion
of
\[
  P_\nu(X)
  :=
  \frac{a_b\bigl(U(X)\bigr)}{a_\delta(X)}
\]
is supported on the partitions $\rho$ satisfying $\nu\subseteq\rho$.

It remains to replace $a_\delta(X)$ by $a_\delta(U(X))$. Put
\[
  D_U(X)
  :=
  \frac{a_\delta\bigl(U(X)\bigr)}{a_\delta(X)}
  =
  \prod_{1\leq i<j\leq N}
  \frac{U(x_i)-U(x_j)}{x_i-x_j}.
\]
This is a symmetric formal power series with constant term $1$. Its inverse
$D_U(X)^{-1}$ therefore exists and has a degreewise Schur expansion. The
bialternant formula gives
\[
  s_\nu\bigl(U(X)\bigr)
  =
  \frac{a_b\bigl(U(X)\bigr)}{a_\delta\bigl(U(X)\bigr)}
  =
  P_\nu(X)\,D_U(X)^{-1}.
\]
The Littlewood--Richardson rule gives
$c_{\rho\alpha}^{\sigma}\neq0\Rightarrow\rho\subseteq\sigma$, so every Schur
function occurring in $P_\nu(X)D_U(X)^{-1}$ contains $\nu$. These identities
are compatible with setting $x_{N+1}=0$, since $U(0)=0$. Consequently, the
coefficient of $s_\rho$ in the $\widehat{\Lam}_A$-expansion agrees with its
coefficient in $N$ variables whenever $N\geq\ell(\rho)$, and the support claim
follows.

We now compute the two lowest coefficients, with $N\geq\ell(\nu)+1$. Since
$u_{m,r}=0$ for $m<r$ and $u_{r,r}=1$, the only summand of degree $|b|$ in
\eqref{eq:alternant-substitution} is $(m_1,\ldots,m_N)=b$, with coefficient
$1$. Since $D_U(X)^{-1}$ has constant term $1$, the coefficient of $s_\nu$ in
$s_\nu(U(X))$ is $1$. In degree $|b|+1$, a summand must have
$(m_1,\ldots,m_N)=b+e_{j_0}$ for some $j_0$, where $e_{j_0}$ is the $j_0$-th
unit vector. This summand is nonzero only if $b_{j_0}+1$ is not among the $b_j$, that
is, only if $\nu$ has an addable box in row $j_0$, and then the rearrangement
is trivial and the coefficient is
$u_{b_{j_0}+1,\,b_{j_0}}=b_{j_0}u_2$. Hence
\[
  P_\nu(X)
  =
  s_\nu(X)
  +u_2\sum_{r}(\nu_r+N-r)\,s_{\nu+\square_r}(X)
  +(\text{terms of degree}\geq|\nu|+2),
\]
the sum running over the rows $r$ of $\nu$ with an addable box
$\square_r$. On the other hand, from
$\frac{U(x)-U(y)}{x-y}=1+u_2(x+y)+(\text{degree}\geq2)$ we get
$D_U(X)=1+(N-1)u_2\,p_1+(\text{degree}\geq2)$, hence
$D_U(X)^{-1}=1-(N-1)u_2\,p_1+(\text{degree}\geq2)$. By Pieri's rule
$p_1s_\nu=\sum_r s_{\nu+\square_r}$, so the coefficient of
$s_{\nu+\square_r}$ in $P_\nu(X)D_U(X)^{-1}$ equals
\[
  u_2\bigl(\nu_r+N-r\bigr)-u_2\bigl(N-1\bigr)
  =
  \bigl(\nu_r+1-r\bigr)u_2
  =
  c\,u_2,
\]
independently of $N$, as claimed.
\end{proof}

Recall from \eqref{eq:fdbSchurconstants} the coefficients
$\Delta_{\!F}(s_\lambda)=\sum_\nu F_{\lambda\nu}\otimes s_\nu$ of the Fa\`a di
Bruno coproduct. For a skew shape we write likewise
$\Delta_{\!F}(s_{\lambda/\mu})=\sum_\nu F_{\lambda/\mu,\nu}\otimes s_\nu$.

\begin{thm}\label{thm:subdiagram-support}
For every skew shape $\lambda/\mu$ and every partition $\nu$, one has
\begin{equation}\label{eq:universalsupport}
  F_{\lambda/\mu,\nu}=0
  \quad\text{unless}\quad
  \nu\subseteq\lambda,
\end{equation}
that is,
$\Delta_{\!F}(s_{\lambda/\mu})\in\sum_{\nu\subseteq\lambda}\Lam\otimes\Q\,s_\nu$.
\end{thm}

\begin{proof}
Let $\varphi:\Lam\to A$ be a character with values in a commutative
$\Q$-algebra $A$, and set
\[
  T_\varphi:=(\varphi\otimes\mathrm{id})\circ\Delta_{\!F}:\Lam\longrightarrow\Lam_A,
  \qquad
  H_\varphi(t):=\sum_{n\geq0}\varphi(h_n)t^n,
  \qquad
  U_\varphi(t):=\frac{t}{H_\varphi(t)^2},
\]
so that $U_\varphi(t)\in t+t^2A[[t]]$. Applying $\varphi\otimes\mathrm{id}$
to \eqref{eq:fdbcoprod} gives
\begin{equation}\label{eq:Tphigen}
  \sum_{n\geq0}T_\varphi(h_n)\,t^n
  =
  \sum_{k\geq0}t^k\,H_\varphi(t)^{1-2k}\,h_k
  =
  H_\varphi(t)\sum_{k\geq0}h_k\,U_\varphi(t)^k .
\end{equation}
Let $\bar T_\varphi:\Lam\to\Lam_A$ be the algebra homomorphism determined by
\[
  \sum_{k\geq0}\bar T_\varphi(h_k)\,t^k
  :=
  \sum_{k\geq0}h_k\,U_\varphi(t)^k,
\]
which makes sense because the coefficient of $t^k$ on the right only involves
$h_0,\ldots,h_k$. Since the usual coproduct satisfies
$\Delta(h_n)=\sum_{a+b=n}h_a\otimes h_b$, identity \eqref{eq:Tphigen} says
that the two algebra homomorphisms $T_\varphi$ and
$(\varphi\otimes\bar T_\varphi)\circ\Delta$ agree on every $h_n$. Hence
\begin{equation}\label{eq:universalfactorization}
  T_\varphi=(\varphi\otimes\bar T_\varphi)\circ\Delta .
\end{equation}
Taking logarithms in the definition of $\bar T_\varphi$ and using
$H(t)=\exp\bigl(\sum_{j\geq1}p_jt^j/j\bigr)$ from \eqref{eq:newton},
\begin{equation}\label{eq:barTp}
  \bar T_\varphi(p_m)
  =
  \sum_{j=1}^{m}\frac mj\,\bigl([t^m]\,U_\varphi(t)^j\bigr)\,p_j .
\end{equation}
Each $\bar T_\varphi(p_m)$ is a linear combination of primitive power sums.
Thus $\bar T_\varphi$ is also a coalgebra homomorphism for $\Delta$, and its
$A$-linear extension to $\Lam_A$ has a Hall adjoint
$\bar T_\varphi^{\,\perp}:\Lam_A\to\widehat{\Lam}_A$, defined degreewise by
$\langle\bar T_\varphi^{\,\perp}(f),g\rangle=\langle f,\bar
T_\varphi(g)\rangle$. It is an algebra homomorphism, because the Hall pairing
identifies the multiplication of $\Lam_A$ with the adjoint of $\Delta$.
Moreover, $\bar T_\varphi$ preserves the length of power-sum monomials, so
$\bar T_\varphi^{\,\perp}(p_j)$ is a (completed) linear combination of the
power sums $p_m$. Using $\langle p_m,p_j\rangle=m\delta_{mj}$ and
\eqref{eq:barTp} we obtain
$\bar T_\varphi^{\,\perp}(p_j)
=\sum_{m\geq j}\bigl([t^m]U_\varphi(t)^j\bigr)p_m
=\sum_{i\geq1}U_\varphi(x_i)^j$. Since the power sums generate $\Lam$ and
$\bar T_\varphi^{\,\perp}$ is an algebra homomorphism,
\begin{equation}\label{eq:barT-substitution}
  \bar T_\varphi^{\,\perp}(f)
  =
  f\bigl(U_\varphi(x_1),U_\varphi(x_2),\ldots\bigr)
  \qquad(f\in\Lam).
\end{equation}
By \cref{lem:substitution-support}, the Schur expansion of
$\bar T_\varphi^{\,\perp}(s_\nu)$ is supported on the partitions
$\eta\supseteq\nu$. Therefore,
$\langle\bar T_\varphi(s_\eta),s_\nu\rangle
=\langle s_\eta,\bar T_\varphi^{\,\perp}(s_\nu)\rangle=0$ unless
$\nu\subseteq\eta$. Combining \eqref{eq:universalfactorization} with the
Schur coproduct $\Delta(s_\lambda)=\sum_{\eta\subseteq\lambda}
s_{\lambda/\eta}\otimes s_\eta$ gives
\[
  T_\varphi(s_\lambda)
  =
  \sum_{\eta\subseteq\lambda}\varphi(s_{\lambda/\eta})\,\bar T_\varphi(s_\eta),
\]
so every Schur function occurring in $T_\varphi(s_\lambda)$ is indexed by a
partition $\nu$ with $\nu\subseteq\eta\subseteq\lambda$. For a skew shape,
$s_{\lambda/\mu}=\sum_\eta c^\lambda_{\mu\eta}s_\eta$ with
$c^\lambda_{\mu\eta}=0$ unless $\eta\subseteq\lambda$, and the same support
bound follows. Thus, for every character $\varphi$ into every commutative
$\Q$-algebra,
\[
  \varphi\bigl(F_{\lambda/\mu,\nu}\bigr)
  =
  \bigl\langle T_\varphi(s_{\lambda/\mu}),s_\nu\bigr\rangle
  =0
  \qquad(\nu\not\subseteq\lambda).
\]
Finally take $A=\Lam$ and $\varphi=\mathrm{id}$, the tautological character:
then $T_\varphi=\Delta_{\!F}$ and $\varphi(F_{\lambda/\mu,\nu})
=F_{\lambda/\mu,\nu}$, which proves \eqref{eq:universalsupport}.
\end{proof}

\begin{cor}\label{cor:winding-support}
For every character $\varphi:\Lam\to A$ and every skew shape $\lambda/\mu$,
the winding operator $T_\varphi=(\varphi\otimes\mathrm{id})\circ\Delta_{\!F}$
satisfies
$\langle T_\varphi(s_{\lambda/\mu}),s_\nu\rangle=0$ unless
$\nu\subseteq\lambda$. In particular, by \cref{thm:fdb},
\[
  d_{\lambda/\mu}(\nu)=0
  \quad\text{and}\quad
  \bigl\langle\theta^{-1}(s_{\lambda/\mu}),s_\nu\bigr\rangle=0
  \qquad\text{unless}\quad
  \nu\subseteq\lambda:
\]
both the expansion \eqref{eq:naturalexp} of the Schur Eisenstein series in
Schur MacMahon series and its inverse are supported on the subdiagrams of
$\lambda$.
\end{cor}

\begin{proof}
Apply $\varphi$ to \eqref{eq:universalsupport}, and take $\varphi=\Phib$
resp.\ $\varphi=\Phib^{\star-1}$, using $\theta=T_\beta$ and
\eqref{eq:thetainversewinding}.
\end{proof}

The extreme layers of \eqref{eq:naturalexp} are completely understood. In top
degree, $d_{\lambda/\mu}(\nu)=c^\lambda_{\mu\nu}$ for $|\nu|=|\lambda/\mu|$.
In degree zero, $d_{\lambda/\mu}(\varnothing)=\beta(\lambda/\mu)$, and this
never vanishes: by \eqref{eq:beta} and \eqref{eq:schurmzv},
\[
  \beta(\lambda/\mu)
  =
  \frac{(-1)^{|\lambda/\mu|}}{(2\pi)^{2|\lambda/\mu|}}\,
  \zeta\bigl(\{2\}^{\lambda/\mu}\bigr),
  \qquad
  \zeta\bigl(\{2\}^{\lambda/\mu}\bigr)
  =
  s_{\lambda/\mu}\Bigl(\frac1{1^2},\frac1{2^2},\frac1{3^2},\dots\Bigr)>0 .
\]
The next result settles the layer $|\lambda/\nu|=1$ as well, in closed form.

\begin{prop}\label{prop:onebox}
Let $\nu$ be a partition, let $\square$ be an addable box of $\nu$ of content
$c$, and let $\nu+\square$ denote the partition obtained by adding it. Then
\[
  F_{\nu+\square,\nu}=(1-2c)\,h_1,
\]
and consequently
\[
  d_{\nu+\square}(\nu)=\frac{2c-1}{24},
  \qquad
  \bigl\langle\theta^{-1}(s_{\nu+\square}),s_\nu\bigr\rangle=\frac{1-2c}{24}.
\]
In particular $d_\lambda(\nu)\neq0$ whenever $\nu\subseteq\lambda$ and
$|\lambda/\nu|=1$.
\end{prop}

\begin{proof}
Let $\varphi:\Lam\to A$ be a character. From
$H_\varphi(t)^{-2}=1-2\varphi(h_1)t+O(t^2)$ we get
$[t^2]\,U_\varphi=-2\varphi(h_1)$. By \eqref{eq:barTp}, $\bar T_\varphi$ is
the identity plus terms of lower degree, so
$\langle\bar T_\varphi(s_\nu),s_\nu\rangle=1$. By
\eqref{eq:barT-substitution} together with
\cref{lem:substitution-support},
$\langle\bar T_\varphi(s_{\nu+\square}),s_\nu\rangle
=\langle s_{\nu+\square},\bar T_\varphi^{\,\perp}(s_\nu)\rangle
=-2c\,\varphi(h_1)$. In the factorization
\eqref{eq:universalfactorization} applied to $s_{\nu+\square}$, only
$\eta\in\{\nu,\nu+\square\}$ can pair nontrivially with $s_\nu$, and
$s_{(\nu+\square)/\nu}=h_1$. Hence,
\[
  \bigl\langle T_\varphi(s_{\nu+\square}),s_\nu\bigr\rangle
  =
  \varphi(h_1)\cdot1+1\cdot\bigl(-2c\,\varphi(h_1)\bigr)
  =
  \varphi\bigl((1-2c)h_1\bigr).
\]
Since $\Lam$ is a polynomial ring over $\Q$, whose elements are separated by
$\Q$-valued characters, $F_{\nu+\square,\nu}=(1-2c)h_1$. Specializing,
$\Phib(h_1)=\beta_2=-\tfrac1{24}$ gives the value of $d_{\nu+\square}(\nu)$.
Also, $\Phib^{\star-1}(h_1)=-\Phib(h_1)=\tfrac1{24}$, because
$\Delta_{\!F}(h_1)=h_1\otimes1+1\otimes h_1$ and
$(\Phib\star\Phib^{\star-1})(h_1)=\varepsilon(h_1)=0$. As $2c-1$ is odd, the
last claim follows.
\end{proof}

\begin{rem}\label{rem:barsharp}
Set $\bar\theta:=\bar T_{\Phib}$. By \eqref{eq:universalfactorization} and
\cref{thm:fdb}, one has $\theta=(\Phib\otimes\bar\theta)\circ\Delta$.
The constants $\beta_{2m}$ are essential in \cref{conj:subdiagram} below: the
analogous nonvanishing fails for $\bar\theta$. Indeed, by
\cref{lem:substitution-support} with $u_2=[t^2]U_{\Phib}=\tfrac1{12}$,
\[
  \bigl\langle\bar\theta(s_{\nu+\square}),s_\nu\bigr\rangle=\frac{c}{12},
\]
which vanishes whenever the added box has content $0$, for example for
$(\nu+\square)/\nu=(1)/\varnothing$ or $(2,2)/(2,1)$. In the factorization
$\theta(s_\lambda)=\sum_{\eta\subseteq\lambda}\beta(\lambda/\eta)\,
\bar\theta(s_\eta)$ individual terms therefore do vanish. The conjecture
asserts that the $\beta$-weighted sum never does.
\end{rem}

For a straight shape, computations suggest that the necessary condition of
\cref{cor:winding-support} is also sufficient.

\begin{conj}\label{conj:subdiagram}
For every pair of partitions $\nu\subseteq\lambda$, one has
\[
  d_\lambda(\nu)\neq0,
  \qquad\text{equivalently}\qquad
  \Phib(F_{\lambda\nu})\neq0.
\]
\end{conj}

The conjecture has been verified for every partition $\lambda$ with
$|\lambda|\leq14$. This covers a total of $25302$ pairs
$\nu\subseteq\lambda$. In the same range one also finds
$\langle\theta^{-1}(s_\lambda),s_\nu\rangle\neq0$ for all
$\nu\subseteq\lambda$.

\subsection{Kernels and relations}\label{subsec:kernels}
Both specializations $\Phig,\PhiG:\Lam\to \M$ are surjective by
\cref{prop:RgInMz} and \cref{thm:transition}, and they are far from
injective. The relations are governed by the classical structure of Eisenstein series.
For $m\geq2$ the form $G_{2m}$ is modular, and it is a universal polynomial in
$G_4,G_6$: define $R_m(p_2,p_3)\in\Q[p_2,p_3]$ by $R_2=p_2$, $R_3=p_3$ and the
recursion, for $2m\geq8$,
\begin{equation}\label{eq:eisrec}
  R_m=\frac{12}{(2m+1)(2m-1)(2m-6)}\sum_{\substack{2a+2b=2m\\ a,b\geq 2}}(2a-1)(2b-1)\,R_aR_b,
\end{equation}
so that $G_{2m}=R_m(G_4,G_6)$ for all $m\geq2$, e.g.\
\[
  G_8=\tfrac{6}{7}G_4^2,\qquad
  G_{10}=\tfrac{10}{11}G_4G_6,\qquad
  G_{12}=\tfrac{72G_4^3+50G_6^2}{143}.
\]
The identity $G_{2m}=R_m(G_4,G_6)$ is classical. Analytically, \eqref{eq:eisrec} is
the recursion for the Laurent coefficients of the Weierstrass $\wp$-function obtained
by comparing coefficients in the differential equation $\wp''=6\wp^2-\tfrac{g_2}{2}$,
going back to Hurwitz \cite{Hu}, see also \cite[Eq.~(59.6)]{Rad} (the proof of
\cref{thm:equiv} reruns exactly this argument on the level of the polynomials
$R_m$). A short combinatorial proof was given by Skoruppa \cite{Sk}.

\begin{thm}\label{thm:kernels}
\begin{enumerate}
\item The kernel of $\PhiG$ is the ideal
\[
  \ker\PhiG=\bigl(p_m-R_m(p_2,p_3)\ \big|\ m\geq4\bigr),
\]
and $\PhiG$ induces an isomorphism $\Lam/\ker\PhiG\xrightarrow{\ \sim\ }\M$ sending
$p_1,p_2,p_3$ to $G_2,G_4,G_6$.
\item The kernel of $\Phig$ is $\theta(\ker\PhiG)$, the ideal generated by the elements
\[
  p_m-S_m,\qquad
  S_m:=p_m-\theta(p_m)+R_m\bigl(\theta(p_2),\theta(p_3)\bigr)\in\Q[p_1,\ldots,p_{m-1}]\quad(m\geq4).
\]
\end{enumerate}
\end{thm}

\begin{proof}
(i) By \cref{lem:psipn} and $G_{2m}=R_m(G_4,G_6)$, the displayed ideal $I$ is
contained in $\ker\PhiG$. The quotient $\Lam/I$ is generated by the images of
$p_1,p_2,p_3$, since modulo $I$ every $p_m$ with $m\geq4$ is a polynomial in $p_2,p_3$.
The induced surjection $\Lam/I\to \M$ sends these three generators to $G_2,G_4,G_6$,
which are algebraically independent because $E_2,E_4,E_6$ are (see e.g.~\cite{Za}).
A surjection from a $\Q$-algebra generated by three elements onto a polynomial ring in
three variables mapping generators to generators is an isomorphism, so $I=\ker\PhiG$.

(ii) From $\PhiG=\Phig\circ\theta$ and the bijectivity of $\theta$ we get
$\ker\Phig=\theta(\ker\PhiG)$, with generators
$\theta\bigl(p_m-R_m(p_2,p_3)\bigr)=\theta(p_m)-R_m(\theta(p_2),\theta(p_3))$. Since
$p_m-\theta(p_m)$ only involves $p_1,\dots,p_{m-1}$ by \eqref{eq:thetap}, and since
$R_m(\theta(p_2),\theta(p_3))\in\Q[p_1,p_2,p_3]$, these generators have the
displayed form with $S_m\in\Q[p_1,\ldots,p_{m-1}]$.
\end{proof}

\begin{ex}\label{ex:kerphi}
For $m=4$ \cref{thm:kernels} gives, after applying $\Phig$, the $q$-series
identity
\[
  P_4=\frac{6}{7}\Bigl(P_2+\frac{P_1}{6}+\frac{1}{1440}\Bigr)^{2}
      -\frac13P_3-\frac1{40}P_2-\frac1{5040}P_1-\frac1{2419200}
\]
for the power sums $P_r=\Phig(p_r)=\sum_m x_m^r$ of \eqref{eq:Pr}. On the
$\PhiG$-side, relations in $\ker\PhiG$ produce linear relations among Schur Eisenstein
series of equal weight. The first one appears in weight $8$,
\[
  \pes{4}=\pes{3,1}+12\,\pes{2,2}-13\,\pes{2,1,1}+13\,\pes{1,1,1,1},
\]
which expresses that
$s_{(4)}-s_{(3,1)}-12\,s_{(2,2)}+13\,s_{(2,1,1)}-13\,s_{(1,1,1,1)}\in\ker\PhiG$.
\end{ex}

\section{The $\sltwo$-structure of Schur Eisenstein series}\label{sec:sl2}
For the formal counterparts of multiple Eisenstein series, an $\sltwo$-algebra
structure was constructed by the first author and van Ittersum \cite{BvIM}. The
results of this section can be seen as a classical realization of such a
structure on the all-$2$ family.

\subsection{The $\sltwo$-triple on quasimodular forms}
Recall that an $\sltwo$-algebra is an algebra together with derivations
$\GD,\GW,\Gdelta$ satisfying
\[
  [\GW,\GD]=2\GD,\qquad [\GW,\Gdelta]=-2\Gdelta,\qquad [\Gdelta,\GD]=\GW.
\]
The ring $\M=\Q[E_2,E_4,E_6]$ is an $\sltwo$-algebra in a classical way (see
e.g.~\cite{Za}).

\begin{lem}\label{lem:sl2M}
Let $\GD=q\frac{d}{dq}$, let $\GW$ be the weight grading ($\GW f=kf$ for
$f\in \M_k$), and let
\[
  \Gdelta := 12\,\frac{\partial}{\partial E_2} = -\frac12\,\frac{\partial}{\partial G_2},
\]
the partial derivative taken in the presentation $\M=\Q[E_2,E_4,E_6]$. Then
$(\GD,\GW,\Gdelta)$ is an $\sltwo$-triple of derivations of $\M$.
\end{lem}

\begin{proof}
By the Ramanujan identities
\[
  \GD E_2=\frac{E_2^2-E_4}{12},\qquad
  \GD E_4=\frac{E_2E_4-E_6}{3},\qquad
  \GD E_6=\frac{E_2E_6-E_4^2}{2},
\]
$\GD$ is a derivation of $\M$ raising the weight by $2$, and $\Gdelta$ is a derivation
lowering it by $2$, so $[\GW,\GD]=2\GD$ and $[\GW,\Gdelta]=-2\Gdelta$. The commutator
$[\Gdelta,\GD]$ is again a derivation, and on the generators
\[
  [\Gdelta,\GD]\,E_2=12\,\frac{\partial}{\partial E_2}\frac{E_2^2-E_4}{12}=2E_2,\quad
  [\Gdelta,\GD]\,E_4=4E_4,\quad
  [\Gdelta,\GD]\,E_6=6E_6,
\]
so $[\Gdelta,\GD]=\GW$.
\end{proof}

\subsection{The lowering operator removes boxes}
\begin{thm}\label{thm:delta}
For every partition $\lambda$,
\[
  \Gdelta\,G(\lambda)=-\frac12\sum_{\mu=\lambda-\square}G(\mu),
\]
the sum running over all partitions obtained from $\lambda$ by removing a single box.
\end{thm}

\begin{proof}
We claim that
\begin{equation}\label{eq:intertwine}
  \Gdelta\circ\PhiG=\PhiG\circ\Bigl(-\frac12\,\frac{\partial}{\partial p_1}\Bigr).
\end{equation}
Both sides are $\Q$-linear maps $\Lam\to \M$ satisfying the same twisted Leibniz rule
along $\PhiG$ (for the left side because $\Gdelta$ is a derivation, for the right side
because $\partial/\partial p_1$ is one), so it suffices to compare them on the
generators $p_m$. By \cref{lem:psipn}, $\PhiG(p_1)=G_2=-\tfrac{1}{24}E_2$, so
$\Gdelta\PhiG(p_1)=-\tfrac12=\PhiG\bigl(-\tfrac12\partial p_1/\partial p_1\bigr)$. For
$m\geq2$ the form $\PhiG(p_m)=G_{2m}$ is modular, hence a polynomial in $E_4,E_6$ alone,
so $\Gdelta\PhiG(p_m)=0=\PhiG\bigl(-\tfrac12\partial p_m/\partial p_1\bigr)$. This proves
\eqref{eq:intertwine}, and evaluating at $s_\lambda$ gives the theorem by
\eqref{eq:boxremoval}.
\end{proof}

For a single column the only removable box is the bottom one, so
\cref{thm:delta} specializes to
\begin{equation}\label{eq:deltacolumn}
  \Gdelta\,G_{\{2\}^{l}}=-\frac12\,G_{\{2\}^{l-1}}\qquad(l\geq1),
\end{equation}
which is exactly the renormalized form of the lowering operator conjectured for
multiple Eisenstein series (cf.\ \cite{BvIM}). Skew shapes reduce to the
straight-shape case by \eqref{eq:LR}.

\subsection{The derivative adds boxes}
\begin{thm}\label{thm:sl2-column}
For all $r\geq 0$,
\[
  \GD\,G_{\{2\}^r} = 3\,G_2\,G_{\{2\}^r}-(r+1)(2r+3)\,G_{\{2\}^{r+1}}.
\]
Equivalently, the generating series
$\mathcal G(T)=\sum_{l\geq0}G_{\{2\}^l}T^{2l+1}$ of \eqref{eq:g222intermsofAl}
satisfies the differential equation
\[
  \GD\,\mathcal G = 3G_2\,\mathcal G-\frac12\,\frac{\partial^2}{\partial T^2}\mathcal G .
\]
\end{thm}

\begin{proof}
Set $\mathcal A(u)=\sum_{r\geq0}A_ru^{2r+1}$ with $A_0=1$. Multiplying the
Andrews--Rose recursion (\cref{prop:alderiv}) by $u^{2r+1}$ and summing
over $r\geq0$ gives
\[
  \GD\,\mathcal A
  = 3A_1\mathcal A-\frac12\,\mathcal A_{uu}
    +\frac18\bigl((u\partial_u)^2-1\bigr)\mathcal A,
\]
where we used $\sum_{r\geq0}(r+1)(2r+3)A_{r+1}u^{2r+1}=\tfrac12\mathcal A_{uu}$ and
$\tfrac{r(r+1)}{2}=\tfrac18\bigl((2r+1)^2-1\bigr)$. Now substitute
$u=2\sin\bigl(\tfrac T2\bigr)$, so that $\mathcal G(T)=\mathcal A(u)$ by
\eqref{eq:g222intermsofAl}. From $u_T=\cos\tfrac T2$ and $u_{TT}=-\tfrac u4$ we get
\[
  \mathcal G_{TT}
  = -\frac u4\,\mathcal A_u+\Bigl(1-\frac{u^2}{4}\Bigr)\mathcal A_{uu},
\]
while $\bigl((u\partial_u)^2-1\bigr)\mathcal A=u\mathcal A_u+u^2\mathcal A_{uu}-\mathcal A$.
Substituting both into the recursion above yields
\[
  \GD\,\mathcal G
  = 3A_1\mathcal G-\frac18\,\mathcal G
    +\frac{u}{8}\mathcal A_u+\Bigl(\frac{u^2}{8}-\frac12\Bigr)\mathcal A_{uu}
  = 3\Bigl(A_1-\frac1{24}\Bigr)\mathcal G-\frac12\,\mathcal G_{TT},
\]
and $A_1-\tfrac1{24}=G_2$. Comparing coefficients of $T^{2r+1}$ gives the first form.
\end{proof}

For general shapes the derivative is computed from \cref{thm:jt} by the Leibniz
rule, which yields a determinantal formula.

\begin{thm}\label{thm:sl2-general}
Let $\lambda/\mu$ be a skew shape with $s=\lambda_1$ columns and set
$c_{ij}=\lambda_i'-\mu_j'-i+j$. Then
\[
  \GD\,G(\lambda/\mu)
  = 3s\,G_2\,G(\lambda/\mu)
    -\sum_{i=1}^{s}\det\!\Bigl[\,
      \begin{cases}
        (c_{kj}+1)(2c_{kj}+3)\,G_{\{2\}^{c_{kj}+1}}, & k=i,\\[2pt]
        G_{\{2\}^{c_{kj}}}, & k\neq i
      \end{cases}
    \Bigr]_{1\leq k,j\leq s}.
\]
Equivalently, $\GD\,G(\lambda/\mu)=\PhiG(X_{\lambda/\mu})$ where
\[
  X_{\lambda/\mu}
  = 3s\,p_1\,s_{\lambda/\mu}
    -\sum_{i=1}^{s}\det\!\Bigl[\,
      \begin{cases}
        (c_{kj}+1)(2c_{kj}+3)\,e_{c_{kj}+1}, & k=i,\\[2pt]
        e_{c_{kj}}, & k\neq i
      \end{cases}
    \Bigr]_{1\leq k,j\leq s}\ \in\Lam .
\]
\end{thm}

\begin{proof}
Apply the derivation $\GD$ to the determinant of \cref{thm:jt}, differentiating
one row at a time, and insert the column formula of \cref{thm:sl2-column} in the
differentiated row. Note that some $c_{kj}$ may be negative. Under the
conventions of \cref{thm:jt} both sides of the column formula vanish for these
entries (for $c_{kj}=-1$ because of the factor $c_{kj}+1$), so it may be applied
to every entry. Collecting the terms $3G_2\,G_{\{2\}^{c_{ij}}}$ row by row gives
$3s\,G_2\,G(\lambda/\mu)$, and the shifted terms give the displayed determinants. The
second form follows since $\PhiG(e_l)=G_{\{2\}^l}$ and $\PhiG(p_1)=G_2$.
\end{proof}

For straight shapes the determinants in \cref{thm:sl2-general} can be expanded into a
more explicit symmetric-function formula.  This is also the most transparent form of
the width bound.

\begin{cor}\label{cor:Dexplicit}
Let $\lambda$ be a partition with $s=\lambda_1$ and conjugate
$\lambda'=(\lambda'_1,\ldots,\lambda'_s)$. Put
\[
  c_{ij}=\lambda'_i-i+j,
  \qquad C(r)=(r+1)(2r+3),
\]
and, for $1\leq i,j\leq s$, define the partitions
\[
  \alpha^{(i)}
  =(\lambda'_1+1,\ldots,\lambda'_{i-1}+1,
    \lambda'_{i+1},\ldots,\lambda'_s),
  \qquad
  \gamma^{(j)}=(1^{j-1}).
\]
Then
\begin{align}
  D_\Lam(s_\lambda)
  &=3s\sum_{\nu=\lambda+\square}s_\nu
    -\sum_{i,j=1}^{s}(-1)^{i+j}C(c_{ij})\,
       e_{c_{ij}+1}\,s_{\alpha^{(i)\prime}/\gamma^{(j)\prime}},
       \label{eq:DexplicitLambda}\\
  \GD\,G(\lambda)
  &=3s\sum_{\nu=\lambda+\square}G(\nu)
    -\sum_{i,j=1}^{s}(-1)^{i+j}C(c_{ij})\,
       G_{\{2\}^{c_{ij}+1}}\,
       G\bigl(\alpha^{(i)\prime}/\gamma^{(j)\prime}\bigr).
       \label{eq:DexplicitG}
\end{align}
Here $e_0=G_{\{2\}^0}=1$ and terms with a negative subscript are zero. In particular,
every Schur function occurring in $D_\Lam(s_\lambda)$ has at most $s+1$ columns.
\end{cor}

\begin{proof}
Expand each shifted-row determinant in \cref{thm:sl2-general} along its shifted row.
After deleting row $i$ and column $j$, the remaining minor is
\[
  \det\bigl[e_{\lambda'_k-k+l}\bigr]_{k\ne i,\,l\ne j}
  =s_{\alpha^{(i)\prime}/\gamma^{(j)\prime}}.
\]
Indeed, reindexing the rows and columns after the deletion gives precisely the 
Jacobi--Trudi matrix on the right. The first term is rewritten by Pieri's rule,
$p_1s_\lambda=s_{(1)}s_\lambda=\sum_{\nu=\lambda+\square}s_\nu$, which proves
\eqref{eq:DexplicitLambda}. Applying $\PhiG$ proves \eqref{eq:DexplicitG}. The first
sum has width at most $s+1$. In the second sum the skew Schur function has outer width
$s-1$, and multiplication by $e_{c_{ij}+1}$ adds a vertical strip, so its Schur
expansion has width at most $s$.
\end{proof}

\begin{ex}\label{ex:Dexamples}
Expanding $X_\lambda$ in the Schur basis gives explicit expansions of
$\GD G(\lambda)$ in Schur Eisenstein series of weight $2|\lambda|+2$, for example
\[
  \GD\,G\bigl((2)\bigr) = 6\,G\bigl((3)\bigr)-11\,G\bigl((2,1)\bigr)+4\,G\bigl((1,1,1)\bigr),
\]
\[
  X_{(2,2)}=6\,s_{(3,2)}-26\,s_{(2,2,1)}+4\,s_{(2,1,1,1)}+4\,s_{(1,1,1,1,1)}.
\]
Shapes other than $\lambda+\square$ occur, so $\GD$ is not simply a weighted
box-adding operator in the Schur Eisenstein basis, in contrast with $\Gdelta$.
\end{ex}

The entire $\sltwo$-triple has a particularly simple lift to $\Lam$.

\begin{prop}\label{prop:widthD}
Define derivations of $\Lam$ by
\begin{equation}\label{eq:homogeneous-triple}
  \delta_\Lam=-\frac12\frac{\partial}{\partial p_1},
  \qquad
  W_\Lam=2\sum_{n\geq1}n p_n\frac{\partial}{\partial p_n},
  \qquad
  D_\Lam(e_l)=3p_1e_l-(l+1)(2l+3)e_{l+1}.
\end{equation}
Here the last formula also holds for $l=0$, with both sides equal to zero. Then
$(D_\Lam,W_\Lam,\delta_\Lam)$ is an $\sltwo$-triple and
\begin{equation}\label{eq:homogeneous-intertwining}
  \PhiG\circ D_\Lam=\GD\circ\PhiG,
  \qquad \PhiG\circ W_\Lam=\GW\circ\PhiG,
  \qquad \PhiG\circ\delta_\Lam=\Gdelta\circ\PhiG.
\end{equation}
In power-sum coordinates the raising operator has the closed formula
\begin{equation}\label{eq:Dpower}
  D_\Lam(p_n)=n(2n+3)p_{n+1}-2n\sum_{a=1}^{n}p_a p_{n+1-a}\qquad(n\geq1).
\end{equation}
Equivalently, for $\mathscr C(T):=T E(T^2)=\sum_{l\geq0}e_lT^{2l+1}$,
\begin{equation}\label{eq:Ctriple}
  \delta_\Lam\mathscr C=-\frac{T^2}{2}\mathscr C,
  \qquad
  W_\Lam\mathscr C=(T\partial_T-1)\mathscr C,
  \qquad
  D_\Lam\mathscr C=3p_1\mathscr C-\frac12\partial_T^2\mathscr C.
\end{equation}
Finally,
\[
  D_\Lam(\Filw_n\Lam)\subseteq\Filw_{n+1}\Lam,
  \qquad
  \delta_\Lam(\Filw_n\Lam)\subseteq\Filw_n\Lam .
\]
\end{prop}

\begin{proof}
The first and third identities in \eqref{eq:homogeneous-intertwining} are
\cref{thm:sl2-column,thm:delta}. The middle one follows because a homogeneous
symmetric function of degree $n$ specializes under $\PhiG$ to weight $2n$.
Applying $D_\Lam$ to the Jacobi--Trudi determinant row by row also gives
$D_\Lam(s_{\lambda/\mu})=X_{\lambda/\mu}$ of \cref{thm:sl2-general}.

The elementary formula gives
\[
  D_\Lam E(t)=3p_1E(t)-3E'(t)-2tE''(t).
\]
Taking the logarithmic derivative and using
$\log E(t)=\sum_{n\geq1}(-1)^{n-1}p_nt^n/n$ yields \eqref{eq:Dpower}.
The commutators with $W_\Lam$ follow from the degrees of the
operators. Applying $[\delta_\Lam,D_\Lam]$ to \eqref{eq:Dpower} gives
$[\delta_\Lam,D_\Lam](p_n)=2np_n=W_\Lam(p_n)$, proving the remaining
$\sltwo$ relation on the algebra generators. The three generating identities follow
directly from \eqref{eq:homogeneous-triple}. The width statement for $D_\Lam$ is
\cref{cor:Dexplicit}, and box removal proves the assertion for $\delta_\Lam$.
\end{proof}

There is a second, conjugate triple that is intrinsic to the Schur MacMahon
specialization.

\begin{prop}\label{prop:macmahontriple}
Set
\begin{equation}\label{eq:macmahon-triple}
  \delta_g=\theta\delta_\Lam\theta^{-1},\qquad
  W_g=\theta W_\Lam\theta^{-1},\qquad
  D_g=\theta D_\Lam\theta^{-1}.
\end{equation}
Then $(D_g,W_g,\delta_g)$ is an $\sltwo$-triple and
\begin{equation}\label{eq:macmahon-intertwining}
  \Phig\circ D_g=\GD\circ\Phig,
  \qquad \Phig\circ W_g=\GW\circ\Phig,
  \qquad \Phig\circ\delta_g=\Gdelta\circ\Phig.
\end{equation}
The raising and lowering derivations have the explicit values
\begin{align}
  D_g(e_r)&=3e_1e_r-(r+1)(2r+3)e_{r+1}+\binom{r+1}{2}e_r,
    \qquad(r\geq0),\label{eq:DgElementary}\\
  \delta_g(p_m)&=\frac{(-1)^m((m-1)!)^2}{2(2m-1)!},
    \qquad(m\geq1).\label{eq:deltagPower}
\end{align}
If $\mathscr A(u):=uE(u^2)$ and $T(u):=2\arcsin(u/2)$, then
\begin{align}
  \delta_g\mathscr A(u)&=-\frac12T(u)^2\mathscr A(u),\label{eq:A-deltag}\\
  D_g\mathscr A(u)&=3e_1\mathscr A(u)-\frac12\partial_u^2\mathscr A(u)
       +\frac18\bigl((u\partial_u)^2-1\bigr)\mathscr A(u).\label{eq:A-Dg}
\end{align}
In particular \eqref{eq:DgElementary} is the symmetric-function lift of the
Andrews--Rose recursion \cref{prop:alderiv}.
\end{prop}

\begin{proof}
Conjugation proves the $\sltwo$ relations. Since
$\PhiG=\Phig\circ\theta$, \eqref{eq:homogeneous-intertwining} gives
\eqref{eq:macmahon-intertwining}. The transition identity \eqref{eq:thetaE} reads
\[
  \theta\mathscr C(T)=\mathscr A\bigl(2\sin(T/2)\bigr).
\]
Conjugating the equations in \eqref{eq:Ctriple}, applying the chain rule, and
then writing $T=2\arcsin(u/2)$ gives \eqref{eq:A-deltag} and \eqref{eq:A-Dg}. Coefficient
extraction in the last identity gives \eqref{eq:DgElementary}. Finally, only the
$p_1$-coefficient of the first-kind formula for $\theta^{-1}(p_m)$ contributes to
$\delta_g(p_m)$. The identity
$t(2m,2)=(-1)^{m-1}((m-1)!)^2$ gives \eqref{eq:deltagPower}.
\end{proof}

Thus the homogeneous triple is adapted to $G$, while its winding translate by the
zeta character, $\theta=T_\beta$, is adapted to $g$.




\section{Bases and integral questions}\label{sec:bases}

\subsection{The basis theorem}\label{subsec:basisconj}
A basis of $\M_{2n}$ is given by the monomials $G_2^aG_4^bG_6^c$ with
$a+2b+3c=n$.  The following result shows that the shapes whose parts are at most $3$
give another basis.

\begin{thm}
\label{thm:equiv}
For every $n\geq0$ the following assertions hold.
\begin{enumerate}
\item The family
\[
  \bigl\{\,G(\lambda)\ \big|\ \lambda\vdash n,\ \lambda_1\leq3\,\bigr\}
\]
is a $\Q$-basis of $\M_{2n}$.
\item The family
\[
  \bigl\{\,g(\lambda)\ \big|\ \lambda_1\leq3,\ |\lambda|\leq n\,\bigr\}
\]
is a $\Q$-basis of $F_{\leq2n}\M$.  Equivalently,
$\{g(\lambda)\mid\lambda_1\leq3\}$ is a basis of $\M$ compatible with the weight
filtration.
\end{enumerate}
\end{thm}

The cardinalities in the theorem are forced: a partition
$\lambda=(3^c,2^b,1^a)$ of $n$ records a triple $(a,b,c)$ with
$a+2b+3c=n$, and these triples index the monomial basis
$G_2^aG_4^bG_6^c$ of $\M_{2n}$.  Thus the point is to prove linear independence.

\begin{rem}\label{rem:width}
In terms of the width filtration \eqref{eq:widthfil}, \cref{thm:equiv} states that
the specialization $\PhiG$ restricts to an isomorphism of graded vector spaces
\[
  \PhiG:\Filw_3\Lam\xrightarrow{\ \sim\ }\M,
\]
doubling the degree, and that $\Phig$ restricts to an isomorphism of filtered vector
spaces $\Filw_3\Lam\simeq \M$, matching the subspace of degree at most $n$ with
$F_{\leq2n}\M$. Equivalently, $\ker\PhiG\cap\Filw_3\Lam=0$ and
$\Filw_3\Lam+\ker\PhiG=\Lam$. The width $3$ is optimal in both directions:
$\Filw_2\Lam$ is too small already in degree $3$, where it has dimension
$2<3=\dim\M_6$, and $\Filw_4\Lam$ meets the kernel, the first instance being the
weight-$8$ relation of \cref{ex:kerphi}, which involves $s_{(4)}$.
\end{rem}

\begin{proof}
We first prove the graded assertion.  Put
\[
  x=G_2,\qquad y=G_4,\qquad z=G_6.
\]
By \cref{lem:psipn} and the definition of the Eisenstein polynomials in
\eqref{eq:eisrec},
\[
  \PhiG(p_1)=x,\qquad \PhiG(p_2)=y,\qquad \PhiG(p_3)=z,
  \qquad \PhiG(p_m)=R_m(y,z)\quad(m\geq4).
\]
The key point is the following $2$-adic divisibility:
\begin{equation}\label{eq:Rmeven}
  R_m(y,z)\in2\Z_{(2)}[y,z]\qquad(m\geq4).
\end{equation}
Here $\Z_{(2)}$ is the localization of $\Z$ at the prime ideal $(2)$, so its elements
are rational numbers with odd denominator. We remark that $2$-adic valuations of
the coefficients of Eisenstein series, written as polynomials in $G_4$ and
$G_6$, are studied in recent work of Chiriac--Jorza \cite{CJ}, which might also
be related to \eqref{eq:Rmeven}.

We prove \eqref{eq:Rmeven} directly from the Weierstrass differential equation encoded
by the recursion \eqref{eq:eisrec}.  Set
\[
  a_m=(2m-1)R_m(y,z)\quad(m\geq2)
\]
and form the Laurent series
\begin{equation}\label{eq:formalP}
  P(w)=\frac1{w^2}+2\sum_{m\geq2}a_mw^{2m-2}
      =\frac1{w^2}+6yw^2+10zw^4+\cdots.
\end{equation}
The coefficient of $w^{2m-4}$ in the identity
\begin{equation}\label{eq:formalPdiff}
  P''=6P^2-60y
\end{equation}
is, for $m\geq4$,
\[
  2(2m-2)(2m-3)a_m
  =24a_m+24\sum_{\substack{r+s=m\\r,s\geq2}}a_ra_s.
\]
After substituting $a_j=(2j-1)R_j$, this is exactly \eqref{eq:eisrec}. The
coefficients of $w^{-4}$, $w^0$, and $w^2$ follow from $R_2=y$ and $R_3=z$.
Consequently \eqref{eq:formalPdiff} holds as a formal Laurent-series identity.
Multiplying both sides by $2P'$ and integrating, we have 
\begin{equation}\label{eq:formalPfirst}
  (P')^2=4P^3-120yP-280z.
\end{equation}
Indeed, the derivative of the difference of the two sides is zero by \eqref{eq:formalPdiff}. The constant term of $(P')^2$ is $-160z$ and the constant term of $4P^3-120yP$ is $120z$.

Now put $t=w^2$ and write
\[
  P(w)=w^{-2}Q(t),\qquad
  Q(t)=1+2A(t),\qquad
  A(t)=\sum_{m\geq2}a_mt^m.
\]
Equation \eqref{eq:formalPfirst}, multiplied by $t^3/4$, becomes
\begin{equation}\label{eq:formalQ}
  (tQ'-Q)^2=Q^3-30yt^2Q-70zt^3.
\end{equation}
If $B=tA'-A=\sum_{m\geq2}(m-1)a_mt^m$, substituting $Q=1+2A$ into
\eqref{eq:formalQ} and dividing by $2$ gives
\begin{equation}\label{eq:formalAB}
  -2B+2B^2
   =3A+6A^2+4A^3-15yt^2-30yt^2A-35zt^3.
\end{equation}
In the coefficient of $t^m$, the terms involving $a_m$ linearly are
$-2(m-1)a_m$ in the term $-2B$ on the left-hand side and $3a_m$ in the term $3A$ on the right-hand side.  All other terms involve only
$a_2,\ldots,a_{m-1}$.  Hence \eqref{eq:formalAB} determines $a_m$ recursively by
division by the odd number $2m+1$.  Starting from $a_2=3y$ and $a_3=5z$, induction
therefore gives
\[
  a_m\in\Z_{(2)}[y,z]\qquad(m\geq2).
\]
We may consequently reduce \eqref{eq:formalAB} modulo $2$.  All terms except
$3A$, $-15yt^2$, and $-35zt^3$ disappear. Since $3\equiv-15\equiv-35\equiv1\pmod2$, we obtain
\[
  0=A+yt^2+zt^3\quad\text{in }\mathbb F_2[y,z][[t]].
\]
Thus $A\equiv yt^2+zt^3\pmod2$, so $a_m\in2\Z_{(2)}[y,z]$ for every $m\geq4$.
Since $2m-1$ is a unit in $\Z_{(2)}$, the same holds for
$R_m=a_m/(2m-1)$, proving \eqref{eq:Rmeven}.

We next pass from power sums to monomial symmetric functions. Let $X_1,X_2,\ldots$ be the variables of $\Lam$ and fix nonnegative integers $a,b,c$. The monomial symmetric function $m_{(3^c,2^b,1^a)}$ is defined by
\begin{equation}\label{eq:monomial-abc}
  m_{(3^c,2^b,1^a)} = \sum_{\substack{ 1\leq i_1<\cdots<i_a,\ 1\leq j_1<\cdots<j_b,\ 1\leq k_1<\cdots<k_c\\
    I\cap J=\varnothing,\ I\cap K=\varnothing,\ J\cap K=\varnothing}} \left(\prod_{u=1}^{a}X_{i_u}\right) \left(\prod_{v=1}^{b}X_{j_v}^{2}\right) \left(\prod_{w=1}^{c}X_{k_w}^{3}\right),
\end{equation}
where $I=\{i_1,\ldots,i_a\}$, $J=\{j_1,\ldots,j_b\}$, and $K=\{k_1,\ldots,k_c\}$ denote the corresponding sets of indices, conditions involving an empty list are omitted, and an empty product is equal to $1$. Thus \eqref{eq:monomial-abc} is the sum of all distinct monomials in which $a$ variables have exponent $1$, $b$ variables have exponent $2$, and $c$ variables have exponent $3$.

For positive integers $w_1,\ldots,w_r$, define
\begin{equation}\label{eq:distinct-index-sum}
  D(w_1,\ldots,w_r) :=\sum_{\substack{ n_1,\ldots,n_r\geq1\\
    n_u\neq n_v\ (1\leq u<v\leq r)}} X_{n_1}^{w_1}\cdots X_{n_r}^{w_r},
\end{equation}
set $D(\varnothing):=1$ for the empty list, and put
\begin{equation}\label{eq:Dabc}
  D_{a,b,c} := D\left(\underbrace{1,\ldots,1}_{a}, \underbrace{2,\ldots,2}_{b}, \underbrace{3,\ldots,3}_{c} \right).
\end{equation}
Since for each monomial in \eqref{eq:monomial-abc} the indices attached to the entries of weight $1$, $2$, and $3$ may be ordered in $a!$, $b!$, and $c!$ ways, respectively, we have
\begin{equation}\label{eq:Dmonomial}
  D_{a,b,c} = a!b!c!\,m_{(3^c,2^b,1^a)}.
\end{equation}

For $r\geq2$, the sums in \eqref{eq:distinct-index-sum} satisfy
\begin{equation}\label{eq:Drecursion}
  D(w_1,\ldots,w_r )={}p_{w_r}D(w_1,\ldots,w_{r-1})-\sum_{j=1}^{r-1} D(w_1,\ldots,w_{j-1},w_j+w_r,w_{j+1},\ldots,w_{r-1}).
\end{equation}
Indeed, the first term on the right-hand side sums over pairwise distinct indices $n_1,\ldots,n_{r-1}$ with $n_r$ unrestricted, and for each $j$ the terms satisfying $n_r=n_j$ form the $j$-th summand of the subtracted sum. These cases are pairwise disjoint because $n_1,\ldots,n_{r-1}$ are pairwise distinct, so subtracting them leaves the sum in \eqref{eq:distinct-index-sum}.

Now suppose that $a+b+c\geq1$, the case $a=b=c=0$ being immediate. Repeated application of \eqref{eq:Drecursion}, followed by collecting equal products of power sums, gives
\begin{equation}\label{eq:Dpower-sum-expansion}
  D_{a,b,c}=p_1^ap_2^bp_3^c+\sum_{\substack{\nu\vdash a+2b+3c\\
    \ell(\nu)<a+b+c}}d_{a,b,c}(\nu)p_\nu,\qquad d_{a,b,c}(\nu)\in\Z,
\end{equation}
where $p_\nu:=\prod_{j=1}^{\ell(\nu)}p_{\nu_j}$. At each application of \eqref{eq:Drecursion}, choosing the first term on the right-hand side leaves the current final weight separate, while choosing one of the terms in the sum replaces two current weights by their sum and decreases the length of the list by one. Choosing the first term at every application therefore gives $p_1^ap_2^bp_3^c$, and since no other choice can produce a product with $a+b+c$ power-sum factors, the coefficient of $p_1^ap_2^bp_3^c$ in \eqref{eq:Dpower-sum-expansion} is $1$. The same reasoning shows that whenever $d_{a,b,c}(\nu)\neq0$, the parts of $\nu$ are obtained from the list in \eqref{eq:Dabc} by repeatedly replacing two entries $u,v$ by the single entry $u+v$, and that for every term in the sum in \eqref{eq:Dpower-sum-expansion} at least one such replacement occurs.

Since $\PhiG$ is an algebra homomorphism, applying it to \eqref{eq:Dpower-sum-expansion} and using $\PhiG(p_1)=x$, $\PhiG(p_2)=y$, and $\PhiG(p_3)=z$ gives
\begin{equation}\label{eq:PhiG-Dpower-sum-expansion}
  \PhiG(D_{a,b,c})= x^ay^bz^c+\sum_{\substack{\nu\vdash a+2b+3c\\
    \ell(\nu)<a+b+c}} d_{a,b,c}(\nu)\prod_{j=1}^{\ell(\nu)} \PhiG(p_{\nu_j}).
\end{equation}
Every summand on the right-hand side belongs to $\Z_{(2)}[x,y,z]$, and by \eqref{eq:Rmeven} it is divisible by $2$ as soon as $\nu_j\geq4$ for some $j$.

Consider now a term of the sum in \eqref{eq:PhiG-Dpower-sum-expansion} for which $d_{a,b,c}(\nu)\neq0$ and every part of $\nu$ is at most $3$, and write $\nu=(3^{c'},2^{b'},1^{a'})$. Since $\ell(\nu)<a+b+c$, at least one replacement of two weights by their sum occurs. Because all final parts are at most $3$ and every replacement increases a weight, such a replacement cannot combine two entries that are both at least $2$, so it must use at least one of the original entries equal to $1$. As a final part equal to $1$ can only come from an original entry equal to $1$ that was never combined with another entry, we conclude $a'<a$. Each replacement preserves the sum of the weights, so $a'+2b'+3c' = a+2b+3c$ and
\[
  \prod_{j=1}^{\ell(\nu)}\PhiG(p_{\nu_j}) =  x^{a'}y^{b'}z^{c'}.
\]
Collecting the terms in \eqref{eq:PhiG-Dpower-sum-expansion} that survive modulo $2$ therefore gives
\begin{equation}\label{eq:Dtriangular}
  \PhiG(D_{a,b,c}) \equiv  x^ay^bz^c +\sum_{\substack{ a',b',c'\geq0\\
    a'<a\\
    a'+2b'+3c'=a+2b+3c}}\gamma_{a',b',c'}x^{a'}y^{b'}z^{c'} \pmod{2},\qquad\gamma_{a',b',c'}\in\mathbb F_2.
\end{equation}

Now fix $n\geq1$ and index both the polynomials $\PhiG(D_{a,b,c})$ and the monomials $x^ay^bz^c$ by the triples $(a,b,c)$ satisfying $a+2b+3c=n$, ordering the rows and columns in the same way by decreasing $a$, with ties ordered arbitrarily. By \eqref{eq:eisrec}, the polynomial $R_m(y,z)$ is homogeneous of weighted degree $m$ when $x,y,z$ have degrees $1,2,3$, respectively, so every $\PhiG(D_{a,b,c})$ occurring here is a linear combination of the monomials $x^ay^bz^c$ of weighted degree $n$. Equation \eqref{eq:Dtriangular} shows that the reduction modulo $2$ of this square coefficient matrix is upper unitriangular, since the diagonal term in the row indexed by $(a,b,c)$ is $x^ay^bz^c$ while every other term in that row has first exponent strictly smaller than $a$. The determinant is therefore congruent to $1$ modulo $2$, and since the entries of the matrix belong to $\Z_{(2)}$, it is a unit in $\Z_{(2)}$ and in particular nonzero over $\Q$. As the monomials $x^ay^bz^c$ with $a,b,c\geq0$ and $a+2b+3c=n$ form a $\Q$-basis of $\M_{2n}$, it follows that the polynomials $\PhiG(D_{a,b,c})$ with $a+2b+3c=n$ form one as well. Applying $\PhiG$ to \eqref{eq:Dmonomial} gives $\PhiG(D_{a,b,c}) = a!b!c!\, \PhiG\bigl(m_{(3^c,2^b,1^a)}\bigr)$, and since $a!b!c!$ is nonzero in $\Q$, rescaling these basis elements shows that
\begin{equation}\label{eq:monomial-image-basis}
  \left\{ \PhiG\bigl(m_{(3^c,2^b,1^a)}\bigr)  \ \middle|\  a,b,c\geq0,\quad a+2b+3c=n \right\}
\end{equation}
is a $\Q$-basis of $\M_{2n}$. The same conclusion holds for $n=0$, since $D_{0,0,0}=m_{\varnothing}=1$.

It remains only to replace monomial symmetric functions by Schur functions.  The
Kostka expansion is
\[
  s_\lambda=\sum_{\mu\vdash n}K_{\lambda\mu}m_\mu,
\]
where $K_{\lambda\mu}\ne0$ only if $\lambda$ dominates $\mu$, and
$K_{\lambda\lambda}=1$.  If $\lambda_1\leq3$ and $K_{\lambda\mu}\ne0$, dominance
implies $\mu_1\leq\lambda_1\leq3$.  Thus the Kostka matrix restricted to the
partitions with largest part at most $3$ is still unitriangular.  Applying $\PhiG$ and
using $\PhiG(s_\lambda)=G(\lambda)$ proves that the $G(\lambda)$ in part~(i) form a
basis of $\M_{2n}$.

Finally \cref{thm:transition}(v) says that the top weight component of
$g(\lambda)$ is $G(\lambda)$.  If a linear relation among the $g(\lambda)$ with
$\lambda_1\leq3$ existed, take the largest size $m$ occurring in it.  Its weight-$2m$
component would be a relation among the basis
$\{G(\lambda)\mid\lambda\vdash m,\ \lambda_1\leq3\}$, so all coefficients of size
$m$ would vanish, and descending on $m$ proves independence.  The cardinality count above,
summed over $m\leq n$, proves part~(ii).
\end{proof}

\subsection{The integral spanning conjecture}\label{subsec:integralspan}
\cref{prop:RgInMz} shows that the Schur MacMahon family has the
correct rational span. \cref{mainconj:integralspan} asks whether every integral
quasimodular form lies in the integral Schur MacMahon span,
\[
  \Rg=\Mz,
  \qquad\text{equivalently}\qquad
  \Mz=\sum_{\lambda}\Z\,g(\lambda)=\Z[A_1,A_2,A_3,\dots].
\]

This is strictly stronger than rational generation: it asserts that $\Rg$ captures
\emph{all} divided congruences inside $\M$. For instance the
basic divided congruence $\tfrac{1-E_2}{24}=P_1=A_1$ lies in $\Rg$, and
\[
  \frac{E_4-E_2^2}{288}
  = q\frac{d}{dq}\!\left(\frac{1-E_2}{24}\right)
  = P_1+5P_2-2P_1^2
\]
is also captured. Note that $\Rg$ is stable under $q\frac{d}{dq}$ by
\cref{prop:alderiv}, as is $\Mz$, so the conjecture is compatible with the
derivation structure.

We also mention the work of Craig \cite{Cr}, who showed that quasimodular forms
with integral Fourier coefficients admit integral representations in
substantially larger algebras of $q$-multiple zeta values.
\cref{mainconj:integralspan} predicts that the much smaller Schur MacMahon
family already captures the entire integral lattice.

\subsubsection{A prime-local criterion}

The conjecture can be stated without referring to a choice of basis or to bounded
computations. Put
\[
  I_g^{\Z}:=\ker\bigl(\Phig:\Lam_\Z\longrightarrow\Z[[q]]\bigr)
  =\Lam_\Z\cap\ker(\Phig:\Lam\longrightarrow \M),
\]
where the rational kernel is described explicitly in \cref{thm:kernels}. For a prime
$p$, let
\[
  \overline I_{g,p}:=(I_g^{\Z}+p\Lam_\Z)/p\Lam_\Z
  \ \subseteq\ \Lam_{\mathbb F_p},\qquad
  K_p:=\ker\bigl(\overline{\Phig}:\Lam_{\mathbb F_p}\to\mathbb F_p[[q]]\bigr).
\]
The inclusion $\overline I_{g,p}\subseteq K_p$ is automatic. The conjecture is
equivalent to equality here.

\begin{prop}\label{prop:integralcriterion}
The following are equivalent.
\begin{enumerate}
\item $\Rg=\Mz$.
\item For every prime $p$, if $f\in\Mz$ and $pf\in\Rg$, then $f\in\Rg$.
\item For every prime $p$,
\begin{equation}\label{eq:kernelcriterion}
  K_p=\overline I_{g,p}.
\end{equation}
\end{enumerate}
In concrete terms, \eqref{eq:kernelcriterion} says that whenever
$f\in\Lam_\Z$ satisfies $\Phig(f)\in p\Z[[q]]$, there are $h\in\Lam_\Z$ and
$i\in I_g^{\Z}$ such that $f=ph+i$.
\end{prop}

\begin{proof}
Suppose first that $\Rg=\Mz$, and let $f\in\Lam_\Z$ have
$\Phig(f)\in p\Z[[q]]$. Then $\Phig(f)/p\in\Mz=\Rg$, so there is
$h\in\Lam_\Z$ with $\Phig(h)=\Phig(f)/p$. Hence $f-ph\in I_g^{\Z}$, proving
$K_p\subseteq\overline I_{g,p}$ and therefore equality.

Conversely, assume \eqref{eq:kernelcriterion} for every prime and take $F\in\Mz$.
Since $\Rg\otimes\Q=\M$ by \cref{prop:RgInMz}, there are $f\in\Lam_\Z$ and an
integer $N\geq1$ with $\Phig(f)=NF$. If $p\mid N$, then
$\Phig(f)\in p\Z[[q]]$, so \eqref{eq:kernelcriterion} gives $f=ph+i$ with
$h\in\Lam_\Z$ and $i\in I_g^{\Z}$. Consequently
$\Phig(h)=(N/p)F$. Repeating this division for all prime factors of $N$ eventually
gives $F\in \Rg$. This proves (iii)$\Rightarrow$(i). The equivalence with (ii) is the
same argument stated in the image rather than in the source.
\end{proof}

\subsubsection{The integral depth-one sector}

One part of the conjecture can be settled by integer-valued polynomials. Let
\[
  \mathcal E^{(1)}:=\Q\oplus\sum_{j\geq1}\Q E_{2j}\subset \M,
\]
where every element is a finite sum. Thus $\mathcal E^{(1)}$ is the linear
Eisenstein sector, allowing mixed weights and including $E_2$, but not products.

\begin{prop}\label{prop:depthoneintegral}
One has
\begin{equation}\label{eq:depthoneintegral}
  \mathcal E^{(1)}\cap\Z[[q]]\subseteq
  \Z\oplus\bigoplus_{r\geq1}\Z P_r\subseteq \Rg.
\end{equation}
The same statement holds over $\Z_{(p)}$ for every prime $p$.
\end{prop}

\begin{proof}
Take $H(q)=c+\sum_{n\geq1}a(n)q^n\in\mathcal E^{(1)}\cap\Z[[q]]$.
The Fourier expansions of the Eisenstein series give an odd polynomial
$F(X)\in\Q[X]$ such that
\[
  a(n)=\sum_{d\mid n}F(d).
\]
M\"obius inversion gives
$F(n)=\sum_{d\mid n}\mu(d)a(n/d)\in\Z$ for every positive integer $n$.

For $r\geq1$ put
\[
  C_r(X):=\binom{X+r-1}{2r-1}
  =\frac{X(X^2-1)(X^2-2^2)\cdots(X^2-(r-1)^2)}{(2r-1)!}.
\]
These polynomials form a $\Z$-basis of the odd integer-valued polynomials. Indeed,
$C_r$ vanishes at $0,\pm1,\dots,\pm(r-1)$ and $C_r(r)=1$. Starting with
$u_1=F(1)$ and successively subtracting
$u_rC_r$, where $u_r$ is the value at $r$ of the remainder, gives
$u_r\in\Z$ and eventually
$F=\sum_ru_rC_r$. The same induction works over $\Z_{(p)}$.

Finally, \eqref{eq:Pr} says precisely that
\[
  \sum_{n\geq1}\left(\sum_{d\mid n}C_r(d)\right)q^n=P_r.
\]
Thus $H=c+\sum_ru_rP_r\in \Rg$. Also, $c\in\Z$. This proves
\eqref{eq:depthoneintegral} and its local version.
\end{proof}

Set $F_{\leq k}\Mz:=F_{\leq k}\M\cap\Z[[q]]$. Because $\Mz$ is filtered rather
than graded, it is natural to
state a finite, weight-bounded refinement. From the expression of $g(\lambda)$ in terms of
power sums one has $g(\lambda)\in F_{\leq 2|\lambda|}\Mz$.

\begin{conj}\label{conj:filtered}
There is an explicit function $B(k)$ such that for every $k\geq 0$
\[
  F_{\leq k}\Mz \subseteq \sum_{|\lambda|\leq B(k)}\Z\,g(\lambda).
\]
\end{conj}

Together with $\Rg\subseteq\Mz$, such a statement provides a finite, effective spanning
set for each bounded-weight lattice. The computations below determine the minimal
possible $B(k)$ for $k\leq16$ and suggest the explicit choice
\begin{equation}\label{eq:Bkguess}
  B(k)=\Bigl\lfloor\frac k2\Bigr\rfloor+\Bigl\lfloor\frac k6\Bigr\rfloor.
\end{equation}

\subsection{Computational evidence}\label{subsec:evidence}
We computed the integral lattices $F_{\leq k}\Mz$ and their Schur MacMahon spans
using exact Fourier expansions. Let $B(k)$ be the least $n$ such that
\[
  F_{\leq k}\Mz\subseteq\sum_{|\lambda|\leq n}\Z\,g(\lambda).
\]
The results are
\[
\begin{array}{c|ccccccc}
  k & 4 & 6 & 8 & 10 & 12 & 14 & 16\\\hline
  \text{index for } n=k/2 & 1 & 3 & 3 & 9 & 81 & 729 & 6561\\
  B(k) & 2 & 4 & 5 & 6 & 8 & 9 & 10 .
\end{array}
\]
The computation used $72$ coefficients. Exact rank shows that these determine the
$67$-dimensional space $F_{\leq20}\M$, so the displayed lattice identities are
rigorous. Consequently \cref{mainconj:integralspan,conj:filtered} hold through weight
$16$.
The values agree with \eqref{eq:Bkguess}, and the obstruction to the naive bound
$B(k)=k/2$ is $3$-primary in every computed case.

\vspace{1cm}
{\bf AI \& computational resource disclosure:} The main results, their formulation, and the underlying idea of proof are the authors' own. ChatGPT~5.6 and Claude Fable~5 were used throughout as research assistants: they carried out and checked computations, and helped with implementing all the objects in SageMath. All statements and proofs were verified by the authors, who take sole responsibility for them.

\end{document}